\documentclass[journal]{IEEEtran}
\usepackage{mathrsfs}
\usepackage{subfigure}
\usepackage[pdftex]{graphicx}
\usepackage{bm,amssymb,graphics,amsthm,amsopn,amstext,amsfonts,color,fullpage,hyperref}
\usepackage[cmex10]{amsmath}
\usepackage{longtable}
\usepackage{subfigure}
\usepackage{fullpage}
\usepackage{xcolor}

\usepackage{cite}
\usepackage[square,sort,comma,numbers]{natbib}
\usepackage{paralist}
\usepackage[]{authblk}
\usepackage{float}
\usepackage[bottom]{footmisc}
\usepackage{graphicx}
\usepackage{lineno}
\usepackage{natbib}
\usepackage{algorithm}
\usepackage[noend]{algpseudocode}
\usepackage{mathtools}

\makeatletter
\def\BState{\State\hskip-\ALG@thistlm}
\makeatother
\usepackage{dblfloatfix}
\usepackage{array}
\newcolumntype{P}[1]{>{\centering\arraybackslash}p{#1}}
\newcolumntype{M}[1]{>{\centering\arraybackslash}m{#1}}

\newcommand{\epc}{\hspace{1pc}}

\def\TheoremsNumberedThrough{%
\theoremstyle{TH}%
\newtheorem{theorem}{Theorem}
\newtheorem{lemma}{Lemma}
\newtheorem{proposition}{Proposition}

\theoremstyle{EX}

\newtheorem{definition}{Definition}

}

\TheoremsNumberedThrough

\usepackage{caption,tabularx,booktabs}

\begin{document}

\title{A Game-theoretic Approach for Dynamic Service Scheduling at Charging Facilities}

\author{Leila~Hajibabai 
	and~Amir~Mirheli 
	\thanks{L. Hajibabai is with the Department of Industrial and Systems Engineering, North Carolina State University, NC 27695, USA. e-mail: lhajiba@ncsu.edu, to whom correspondence should be addressed.}
	\thanks{A. Mirheli is with the Operations Research Program, North Carolina State University, NC 27695, USA. e-mail: amirhel@ncsu.edu.}
}

%

\maketitle
\vspace{-3mm}
\begin{abstract}%
Electric vehicle (EV) charging patterns are highly uncertain in both location, time, and duration particularly in association with the predicted high demand for electric mobility in the future. An EV can be charged at home, at charging stations near highway ramps, or on parking lots next to office buildings, shops, airports, among other locations. Charging time and duration can be fixed and continuous or flexible and intermittent. EV user preferences of charging services depend on many factors (e.g., charging prices, choice of destinations), causing EV charging patterns to shift in real-time. Hence, there is a need for a highly flexible EV charging network to support the rapid adoption of the technology. This study presents a dynamic scheduling scheme for EV charging facilities considering uncertainties in charging demand, charger availability, and charging rate.
The problem is formulated as a dynamic programming model that minimizes the travel and waiting costs and charging expenses while penalizing overcharging attempts. An integrated generalized Nash equilibrium technique is introduced to solve the problem that incorporates a Monte Carlo tree search algorithm to efficiently capture the uncertainties and approximate the value function of the dynamic program. Numerical experiments on hypothetical and real-world networks confirm the solution quality and computational efficiency of the proposed methodology. This study will promote EV adoption and support environmental sustainability by helping users lower the charging spot search burden via a real-time, user-adaptive optimizer. Stakeholders can retrieve charger utilization and pricing data and get feedback on their charging network policies.
\end{abstract}


\IEEEpeerreviewmaketitle

\section{Introduction}\label{intro}
\IEEEPARstart{E}{lectric} 
vehicles promise huge benefits to society by supporting environmental sustainability and shifting our energy consumption in the transportation sector toward renewable sources.
Predicted by Bloomberg \citep{finance2020electric}, 55\% of new sales of automobiles world wide will be EVs by 2040. Rapid adoption of EVs requires the development of a highly flexible EV charging network \citep{finance2020electric,cazzola2016global} with a wide variety of charging service options at different locations to satisfy diversified charging requests from EV users.
%
%
Recent ongoing government and industry plans to deploy charging infrastructure are promising and particularly, electric utilities support the charging network expansions \citep{outlook2019electric}. 
However, the market size of charging infrastructure grows much slower than the charging demand.
%
On the other hand, recent advancements in battery technology has improved EV driving range but the charging rate of EVs still remains slow.
For instance, typical EVs with medium-sized battery (i.e., $16-24\, kWh$) are fully charged within several hours with level 1 chargers. Fully charging contemporary EVs with $\sim 50\, kWh$ battery would still take $\sim 30\, min$, even with Level 3 fast chargers that benefit from direct-current (DC).
The majority of current EVs are not capable of DC fast charging due to the connector types and need to park in charging facilities for hours to recharge \citep{duke}. 
Long-term parking and charging periods, especially with a static cost for each charging attempt, can accumulate unserved EV users, particularly in high-demand areas.
Therefore, the charging network needs to be significantly expanded and properly managed to sustain the transition to EVs. 
The charging facility scheduling problem is highly dynamic considering the stochastic nature of EVs' state-of-charge (SOC) and travel plan, charger availability, and charging rate over time and space.
It is beneficial to develop a user-adaptive framework to dynamically schedule the utilization of existing charging facilities. 
The optimal schedule shall satisfy the charging demand while minimizing the travel, charging, and waiting costs of each individual user.   
The optimal charger allocation plans will help share real-time data on the occupancy and waiting time of charging facilities and expand their usage.
Figure \ref{fig:intro} presents various factors that affect EV user decisions with respect to charger choices. 
As indicated, two charger types at various locations, each with a different pricing scheme, are offered. Users choose chargers that match their criteria (e.g., proximity to destinations, pricing scheme) and minimize their costs based on their travel destinations.
%
%
	%
	\begin{figure}[H]
		\begin{center}
			\includegraphics[height=0.83in]{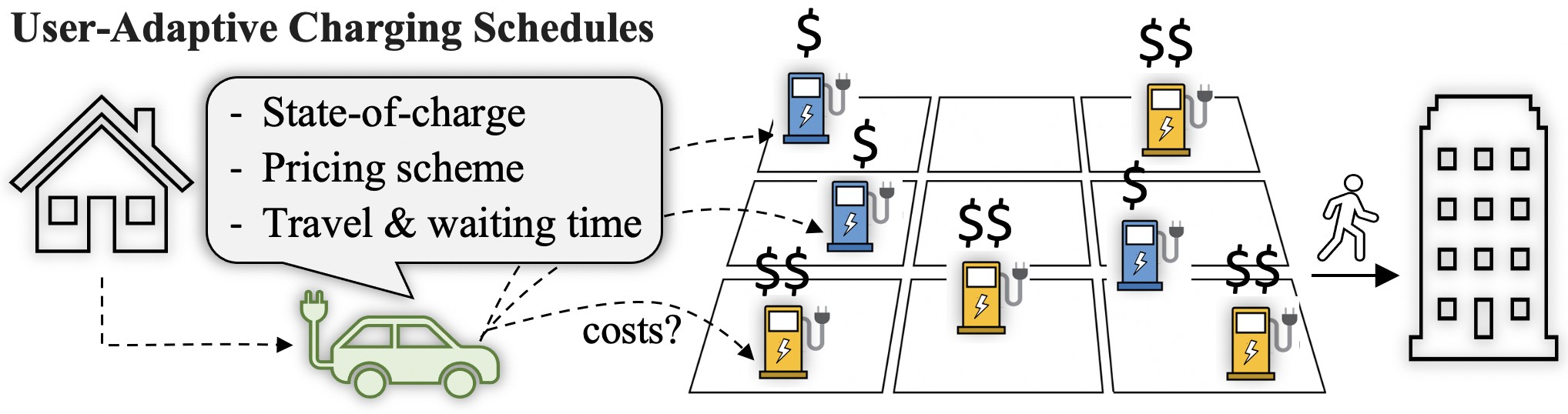}
	\vspace{-20pt}			
	\caption{EV user charging behavior given different charger types and user-centric factors.}
			\label{fig:intro}
		\end{center}
	\end{figure}
	\vspace{-10pt}

This paper proposes a dynamic EV charging scheduling procedure under uncertain charging demand, charger availability, and charging rate. The proposed problem is formulated as a DP with dynamic user behavior and mixed-integer decisions that minimizes the total user costs, including (a) travel costs and waiting time to start the service and (b) charging expenses. The model aims to find the optimal charging duration and spot assignment for each EV user over the planning horizon, given their subsequent travel plans and reactions toward charging prices and waiting time to get served.
A generalized Nash equilibrium (GNE) based technique, coupled with a Monte Carlo tree search (MCTS), with an embedded consensus-based coordination and shooting heuristic is implemented to determine the optimal schedule of EV chargers (i.e., the required number of charging time periods) in parking areas and optimally assign users to charging spots at each time period. 
In particular, the methodology employs a stochastic look-ahead technique that first models a GNE to distribute the problem to EV user-level models. A consensus-based coordination scheme is incorporated into the GNE procedure to push the user-level solutions toward system-level optimality and find near-optimal solutions.
Then, an MCTS algorithm, with an embedded tree policy and shooting heuristic,
is implemented to reduce the expanded search and approximate the value function of the dynamic program. The tree policy evaluates the available actions and estimates the value functions over time, while the
shooting heuristic predicts the value of recently added tree nodes to the tree to determine optimal charging schedules (i.e., number of time periods to charge and charging spot assignments at each time period).
The numerical results reveal that the proposed algorithm can solve the problem efficiently.  

The exposition of the papers follows. Section \ref{sec:litRev} will review the existing literature on relevant service facility logistics. Section \ref{sec:modelformulation} presents the model formulation. Section \ref{sec:Algorithm} illustrates the methodology that integrates GNE and MCTS with embedded tree policy and shooting heuristic. Numerical results are detailed in Section \ref{sec:numericalresults}. And, Section \ref{sec:conclusion} summarizes the concluding remarks.

\vspace{-1mm}
\section{Literature Review}\label{sec:litRev}
This section summarizes the literature on (i) park \& charge scheduling, (ii) capturing stochastic effects, (iii) dynamic control, and (iv) game-theoretic strategies, particularly in the context of EV charging facility logistics.

\vspace{-3mm}
\subsection{Park \& Charge Scheduling}
Literature has shown park-and-charge studies that incorporate charging facilities into parking lots to avoid charging complexities within trip chains. In particular, EVs get charged while parked at a destination. For instance, \cite{kuran2015smart} have proposed an optimal recharge scheduling scheme for parking areas on a daily basis to identify charging spot, time, and amount based on EVs' arrival and departure time, SOC, and travel range.
%
%
The proposed model aims to maximize the total (i) parking lot revenue  and (ii)  number of served EVs (whose charging requirements are satisfied by the departure time). 
The model is formulated as a two-layered problem, where each layer is optimized independently.  
%
Similarly, \cite{huang2012scheduling} have implemented various scheduling algorithms (i.e., first come first serve, earliest deadline first, shortest job first, and longest job first) to design plug-in hybrid EV (PHEV) charging schedules, where earliest deadline first outperforms the rest. 
The study assumes a fixed charging rate to schedule PHEV chargers. 
%
%
\cite{yao2016real} have presented a real-time charge scheduling scheme that accommodates demand response programs in a parking facility. The response programs restore the balance between charging demand and supply in a smart grid and optimize the charging schedule. The optimization aims to simultaneously (i) maximize the number of EVs selected for charging at each time period and (ii) minimize the EVs' utility payments by choosing proper time slots to charge. The study applies an on/off strategy for the charge scheduling as a binary optimization problem. A hybrid technique is implemented that combines a linear programming and a modified convex relaxation. The proposed approach requires a significant deployment cost to install chargers at all spots and apply the on/off strategy. 
Furthermore, \cite{he2021charging} has developed a fleet charging and repositioning framework to satisfy the shared electrified mobility demand. Given the location and capacity of charging facilities by the network operator, the number of relocated EVs will be determined within the region to address the asymmetric demand. The problem is formulated as a queuing-location model that aims to maximize the operator's expected annual profit while the facility deployment and EV repositioning costs are minimized. 
The study does not incorporate the dynamic charging policy of each EV user in their approach. 

\vspace{-3mm}
\subsection{Capturing Stochastic Effects}
A number of studies in the literature have evaluated possible uncertainties in EV charging schedules. For example,  
\cite{wu2017two} have proposed a two-stage stochastic optimization model to schedule EV charging under uncertain EV arrival time, charging price, and charging demand upon arrival. The study aims to optimize the charging loads by minimizing the expected operational costs and the number of unserved users within a fixed time-window captured from the arrival time of each EV to the charging station. They have applied a Monte Carlo based sample-average approximation technique and an L-shaped method to solve the problem. The proposed model does not account for the uncertainty in charging rates as it may increase the number of unserved users due to insufficient charge for their next trips. 
%
%
\cite{lu2018optimal} have proposed a multi-layer time-space network to schedule a fleet of taxis, including EVs and internal combustion engine vehicles with advanced reservations. The problem is formulated as an integer multi-commodity network flow assuming uncapacitated charging stations that require significant capital investments. The study assumes a fixed charging time that impedes EV users to have flexible charging schedules. 
%
%
%
On the other hand, \cite{nourinejad2016equilibrium} have further considered the impact of vehicle-to-grid (V2G) technology on power distribution network operators that impose changes in trip schedules when the V2G payments are high. The scheduling problem is formulated as a bi-level optimization program with the equilibrium activity patterns in the lower level as a result of the upper-level pricing and AC power distribution. The problem is solved using the method of successive averages to find the local optimum. The proposed non-convex formulation has imposed challenges on finding system-level optimal solutions. 

\vspace{-4mm}
\subsection{Dynamic Control}
Furthermore, \cite{sweda2017optimal} have developed a recharging policy to determine the optimal amount to recharge at each charging station. A dynamic model that uses forward recursion is proposed to minimize total EV charging costs by reducing the number of stops and avoiding overcharging attempts. This study has implemented heuristic methods for continuous charging along a path with equally-distant charging stations.
Additionally, \cite{sweda2017adaptive} have proposed a dynamic programming (DP) approach to determine the optimal routing and recharging policy in a network with stochastic charging station availability. The proposed model aims to minimize travel time, waiting time, and charging costs. The state variables are the location and SOC of EV users as well as charging station availability at each time period. 
The optimal policy represents the amount of charge needed at the current location and the route to the subsequent charging stop.
The proposed DP-based approach uses an optimal priori policy for a grid network that reduces the computational burden by dropping the station availability state from the value function to simplify finding the feasible paths. However, the problem is still intractable, and suffers from a large state space as the set of feasible paths dynamically changes due to the charging station availability that affects the routing decisions. 
\cite{schneider2014electric} have introduced a recharging station and routing problem under time-window constraints. This study aims to minimize the number of employed EV drivers and the total distance traveled for last-mile delivery carriers. A hybrid heuristic approach integrating a variable neighborhood search and a tabu search is proposed to find near-optimal solutions. 
\cite{tang2019robust} have incorporated the driving range constraints and charging station capacity into a scheduling problem for electric buses. A static model is introduced as a buffer-distance strategy to address the trip time stochasticity and reduce the en-route breakdown rate of electric buses. Additionally, a dynamic model is used to update the road traffic condition and continuously schedule the electric bus fleet. However, the proposed branch and bound is unable to solve large-scale cases, especially for the dynamic scheduling strategy. 

The introduction of uncertainties imposes additional complexity to the problem. Literature has presented research efforts on the implementation of DP approaches with discrete state and action spaces \citep{godfrey2002adaptive,qian2017time}. 
%
%
%
Additionally, existing efforts utilize shooting heuristic (SH), introduced by Newell \citep{newell1993simplified,newell2002simplified}, to solve boundary value problems in DPs \citep{betts2010practical}. This approach aims to reach a targeted final boundary state following a dynamic system. For instance, \cite{murat2016global} have utilized SH in a facility location and capacity acquisition problem on a line with dense demand. The study presents a DP with two-point boundary values and solves it with SH.
\cite{zhou2017parsimonious} have also applied a parsimonious SH algorithm to smoothen the vehicle trajectories moving toward a signalized intersection by controlling their acceleration profiles. In this variation of SH, each infinite-dimensional vehicle trajectory is introduced by a few segments of analytical quadratic curves. The method builds a large number of vehicle trajectories with physical restrictions, traffic signal timing, and car following safety. 
Similarly, \cite{guo2019joint} have proposed a DP-SH algorithm for a problem, including a connected automated vehicle (CAV) trajectory optimization for an intersection control with a mixed traffic of CAVs and human-driven vehicles. The SH generates near-optimal trajectories for vehicles in a platoon, while DP adjusts traffic signals for a given trajectory.  

\vspace{-3mm}
\subsection{Game-theoretic Strategies}
Another category of research has approached the scheduling problems using game-theoretic models. In particular, a generalized Nash equilibrium (GNE) studies the inter-relationships between system users as a Nash game, where the strategy set of each player relies on the other players' strategies \citep{suwansirikul1987equilibrium,facchinei2007generalized,meunier2010equilibrium}.
For example, \cite{wright2010dynamic} have formulated a revenue management strategy for airline alliances as a Markov game model, where the decision of one airline in interline itineraries may result in a sub-optimal revenue for the alliance. While the Markovian transfer price fails to coordinate an arbitrary alliance, the equilibrium acceptance policy is derived using the value functions from each airline's dynamic model. The proposed methodology does not incorporate customer feedback on airline pricing strategy. 
Furthermore, \cite{corman2020interactions} have proposed a microscopic railway traffic optimization model to simultaneously control railway traffic in real-time while minimizing passenger travel time. A Nash equilibrium is achieved given the availability of information on train and customer arrivals (e.g., first come first serve and timetable with no delays). The proposed methodology does not consider the stochasticity involved in information about future railway states. 
Literature has shown the relationship between GNE and quasi-variational inequalities (QVIs) \citep{bensoussan1974points}. For instance, \cite{zhou2005generalized} have proposed a bi-level transit fare equilibrium model, where the upper level represents a non-cooperative game between multiple transit operators, modeled as a QVI, who adjust the fare to maximize their own profits. The lower level formulates a stochastic user equilibrium transit assignment model.

\vspace{-3mm}
\subsection{Summary}\label{sec:Summary}
The aforementioned studies have overlooked the uncertainties in EV charging rates and the availability of charging options for EV users with flexible arrival and departure times. Table \ref{table:literature} summarizes the relevant contributions and research gaps. This study aims to solve the dynamic EV charging scheduling considering the uncertainties in charging demand, charger availability, and charging rate. Given the dynamic user behaviors and mixed-integer decisions, an integrated solution technique is developed that includes (i) a GNE procedure with a consensus-based coordination scheme to find near-optimal charging schedules and (ii) a Monte Carlo tree search algorithm with an embedded SH and a stochastic look-ahead technique to reduced expanded search and approximate the value function of the dynamic program. Next section presents the proposed model formulation.

\begin{center}
\begin{table*}[t]%
\caption{Summary of relevant literature.\label{table:literature}}
\centering
\scriptsize
\begin{tabular*}{450pt}{@{\extracolsep\fill}lcccc l@{\extracolsep\fill}}
 	\hline\hline
\toprule
&\multicolumn{4}{@{}c@{}@{}@{}}{\textbf{Research Components}} &  \\
\textbf{Contribution/Methodology} & \textbf{Park \& Charge}  & \textbf{Stochasticities}  & \textbf{Dynamics} & \textbf{Game-theory} & \multicolumn{1}{l@{}}{\textbf{Authors}}   \\
\midrule
\hline
Recharge scheduling optimization: & \checkmark & & & &\cite{kuran2015smart}\\
\epc independent two layer model&&&&&\\
Comparison of scheduling algorithms:  &\checkmark&&&&\cite{huang2012scheduling}\\
\epc FCFS, earliest deadline, shortest \& longest job&&&&&\\
Real-time demand response model:  &\checkmark&&&&\cite{yao2016real}\\
\epc linear program \& modified convex relaxation&&&&&\\
Fleet charging \& repositioning:  &\checkmark&&&&\cite{he2021charging}\\
\epc queuing-location model&&&&&\\
Two-stage stochastic optimization: Monte Carlo  &\checkmark&\checkmark&&&\cite{wu2017two}\\
\epc approximation \& L-shaped method&&&&&\\
Integer multi-commodity network flow:  &\checkmark&\checkmark&&&\cite{lu2018optimal}\\
multi-layer time-space network&&&&&\\
V2G on power distribution network operators:   &\checkmark&\checkmark&&&\cite{nourinejad2016equilibrium}\\
\epc bi-level models \& method of successive averages&&&&&\\
Continuous charging with equally-distant stations:  &\checkmark&&\checkmark&&\cite{sweda2017optimal}\\
\epc dynamic model \& forward recursion heuristics&&&&&\\
Optimal routing \& recharging policy;  &&\checkmark&\checkmark&&\cite{sweda2017adaptive}\\
\epc stochastic adaptive charging stations&&&&&\\
Recharging station \& routing with time-windows:  &&&\checkmark&&\cite{schneider2014electric}\\
\epc neighborhood search \& tabu search&&&&&\\
Scheduling of electric buses with driving range   &&\checkmark&\checkmark&&\cite{tang2019robust}\\
\epc constraints \& charging station capacity&&&&&\\
Implementation of DP approaches with  &&&\checkmark&&\citep{godfrey2002adaptive}, \citep{qian2017time}\\
\epc discrete state and action spaces&&&&&\\
Shooting heuristic (SH) to solve  &&&\checkmark&&\citep{newell1993simplified}, \citep{newell2002simplified}, \citep{betts2010practical},  \cite{murat2016global}, \cite{zhou2017parsimonious} \\
\epc boundary value problems in DPs&&&&&\\
CAV trajectory optimization: DP \& SH &&&\checkmark&&\cite{guo2019joint}\\
Inter-relationships among users with GNE &&&&\checkmark&\citep{suwansirikul1987equilibrium}, \citep{facchinei2007generalized}, \citep{meunier2010equilibrium}, \cite{wright2010dynamic}, \cite{corman2020interactions}\\
GNE \& quasi-variational inequalities &&&&\checkmark&\citep{bensoussan1974points}, \cite{zhou2005generalized}\\
\bottomrule
\hline
\end{tabular*}
\end{table*}
\end{center}

\vspace{-1mm}
\section{Model Formulation}\label{sec:modelformulation}
\noindent This section proposes a dynamic scheduling plan for EV charging facility utilization that assigns EV users to charging spots based on their arrival time and subsequent travel plans. Currently, charging facilities are able to show the real-time occupancies and offer reservations to users using a website or smartphone app to ensure that a spot is available when a user arrives. The objective is to minimize total user costs including (i) travel time from origins to charging facilities and beyond to final destinations, (ii) waiting time at charging facilities, and (iii) charging costs. Since the charging rate generally slows down as SOC reaches battery capacity, users are discouraged from staying for long hours for a full charge. In fact, EV users are encouraged to leave chargers with enough SOC according to their travel plans so that other users can utilize the chargers. 
The proposed formulation incorporates EV user reactions toward charging price and waiting time at facilities into the dynamic scheduling framework based on users' subsequent travel plans.

Table \ref{table:definitions} summarizes all the notations used in this paper. 
We first introduce the physical and temporal elements of the problem.
Let $T$ be the number of discrete time periods in the planning horizon and $\Gamma=\left\{0,1,\cdots,T-1\right\}$ denote the times at which users make charging decisions. We let $J$ represent the set of physical parking spaces, where each lot $j\in J$ offers a set of charging spots of type $k\in K= \{0,1\}$ over time. A charger type $k=0$ represents a slow charger, while $k=1$ indicates a fast charger. We define $c_{jk}$ to represent the capacity of each charger of type $k\in K$ at parking lot $j \in J$.
Let $\widehat{\mathcal{J}}^t_{jk}$ denote the number of charging spots with type $k$ charger at $j\in J$ that first become available at time $t\in \Gamma$. Accordingly, $\widehat{\mathcal{J}}^t= \bigcup_{j \in J, k \in K}\widehat{\mathcal{J}}^t_{jk}$ represents the spatial distribution of all newly recognized charging spots at time $t$. 
Additionally, $\mathcal{J}^t_{jk}$ denotes the number of charging spots with type $k$ charger already available in parking lot $j$ at time $t$ before any new spot availability. Similarly, $\mathcal{J}^t=\bigcup_{j \in J, k \in K}{\mathcal{J}}^t_{jk}$ represents the total number of charging spots that are already available at time $t$. Hence, $\mathcal{J}^{t^+}$ denotes the total number of available charging spots at time period $t$, where $\mathcal{J}^{t^+}=\mathcal{J}^t+\widehat{\mathcal{J}}^t$ that includes all existing as well as newly recognized charging spots.
We introduce $p_{jk}^{t}$ to represent the charging price of a spot with type $k$ charger in lot $j$ at time $t$.
Accordingly, $\mathcal{P}^t$ represents the set of charging prices at time $t$.

To monitor the EV user activities, we define $\widehat{\mathcal{D}}^t$ for each $t\in \Gamma$ to denote the set of users that first arrive at parking lots (i.e., cruise for available charging spots) in time period $t$. Accordingly, $\mathcal{D}^t$ denotes the existing users at time $t$ before the new EV arrivals, i.e., $\widehat{\mathcal{D}}^t$, are included in the system. 
Similarly, $\mathcal{D}^{t^+}$ indicates the set of available EV users at time $t$, including the new users that just arrived; i.e., $\mathcal{D}^{t^+}=\mathcal{D}^t \cup \widehat{\mathcal{D}}^t$. %
Let $\widehat{\mathcal{B}}^t$ represents the set of SOCs of EV users $i\in \mathcal{D}^t$ that first arrive at time period $t$ who cruise to find an available charging spot given the occupancy $\sigma^t_{jk}$ of charger type $k \in K$ located in their preferred parking lot at $j$. 
Accordingly, $\mathcal{B}^{t}$ denotes the SOC of available users at time $t$. Additionally, $b_i^t$ denotes the SOC of EV user $i$ at time $t$.
We also define $\Psi^{t}$ to illustrate the set of parking durations (i.e., $\psi_i$) of EV users $i\in \mathcal{D}^t$ at time $t$. 
Furthermore, $\widehat{\mathcal{L}}^t_{jk}$ denotes the expected average waiting time of EV users targeting lot $j$ with type $k$ charger at time $t$. Similarly, $\widehat{\mathcal{L}}^{t}=\bigcup_{j \in J, k\in K}\,\widehat{\mathcal{L}}^t_{jk}$ represents the set of newly perceived waiting times at time $t$.
%

%
%
The state of the system $S^t$ at time $t$ is captured by the spatial distribution of available charging spots and accompanying charging prices, distribution of charging demand along with their SOC and parking duration, and expected waiting time to get served; i.e., $S^t =\left\{\mathcal{J}^t,\mathcal{P}^t,\mathcal{D}^t,\mathcal{B}^t,\Psi^{t},\mathcal{L}^t\right\}$.

We define the decision variables as follows. At each time period $t$, EV user $i \in \mathcal{D}^t$ makes a decision $y_{ijk}$ on the location and type of charger they plan to target. In other words, $y_{ijk}=1$ if user $i$ charges with type $k$ charger in lot $j$, or 0 otherwise.
Additionally, $n_{ijk}$ represents the decision variable on maximum allowed charging duration to be determined at time period $t$ based on the users' current SOC and subsequent travel plan. Thus, $n_{ijk}$ defines the charging duration of user $i$ with type $k$ charger in a charging spot at parking lot $j$.

We let $\mathcal{O}$ and $\Delta$ respectively define the set of EV travel origins and destinations. EV user $i \in \mathcal{D}^t$ from origin $o \in \mathcal{O}$ may arrive at a charging spot located in parking lot $j$ at time $t \in \Gamma$ and take another trip to destination $\delta\in \Delta$. Hence, the charging demand is defined by EVs' travel origin and destination, given their arrival time.
EV users experience a driving cost $v_{oj}$ from origin $o$ to charging spot located in parking lot $j$ as well as a driving cost $\mu_{j\delta}$ to final destination $\delta\in \Delta$ to complete their trip remainder. Note that $\mu_{j\delta}$ can also be the driving cost of going back to travel origin (e.g., home). Besides, $\mu_{j\delta}$ can be zero if lot $j$ is close enough to the final destination.
The impact of charging facility occupancy on waiting time is captured as follows \citep{horni2013agent}.
\vspace{-1mm}
\begin{align}
&\mathcal{L}^t_{jk}={e_{jk}\,\beta_{jk}}(1-c_{jk}^{-1}(\sigma^t_{jk}+\sum_{i \in \mathcal{D}^t}y_{ijk}))^{-1},\nonumber\\
&\epc\epc\epc \forall j\in J,\,k \in K,\, t\in \Gamma,\, \label{eq:waiting}
\end{align}
%
%
\noindent where $e_{jk}$ represents the time spent to find an available charging spot of type $k$ at lot $j$. Parameter $\beta_{jk}$ is a constant that represents the reaction of drivers to shared information on charging facility occupancies, where $\beta_{jk}=1$ when drivers are fully aware of the updated occupancies, or 0 otherwise.
The proposed mathematical model optimizes total user costs at time $t$, as
\vspace{-1mm}
\begin{subequations}
    \begin{align}
        &\displaystyle{\operatornamewithlimits{\mbox{min}}_{\boldsymbol{y},\boldsymbol{n}}} \,\,\phi^{t}_{i}(\boldsymbol{y},\boldsymbol{n})=\sum_{k \in K} \sum_{j \in J} \big( \theta (v_{oj} + \mu_{j\delta}) + \theta' \widehat{\mathcal{L}}^t_{jk} \big)y_{ijk} \nonumber\\
        &+ \alpha p_{jk}^{t} n_{ijk} + \alpha'(\psi_i - n_{ijk}) 
        \label{eqref:ObjFunMain}\\[2pt]
		& \mbox{subject to} \nonumber\\
		& \sum_{k \in K} \sum_{j \in J} y_{ijk} \le 1, \forall i \in \mathcal{D}^t,\label{eqref:cons1Main}\\[2pt]
		& n_{ijk} \le My_{ijk},\,\forall i \in \mathcal{D}^t,\, j \in J,\, k \in K,\,\label{eqref:cons2Main}\\[2pt]
		& b_{i}^{t+1} \ge b_{i}^{t} + \pi(k+1) y_{ijk},\forall i \in \mathcal{D}^t, j \in J, k \in K,\label{eqref:cons3Main}\\[2pt]
		& b_{i}^{t} + \sum_{k \in K} \pi(k+1) n_{ijk} \ge Q_{i},\,\forall  i \in \mathcal{D}^t,\,j \in J,\,\label{eqref:cons4Main}\\[2pt]
		& y_{ijk} + \sum_{-i \in \mathcal{D}^t} y_{-i,jk} \le c_{jk} - \sigma_{jk}^{t},\,\forall j \in J, k \in K.\,\label{eqref:cons5Main}
    \end{align}
\end{subequations}
%
%
Objective function $\phi^{t}_{i}(\boldsymbol{y},\boldsymbol{n})$ in \eqref{eqref:ObjFunMain} aims to minimize the total costs imposed to user $i \in \mathcal{D}^t$ at time $t\in \Gamma$ to travel from origins to destinations with charging attempts en-route. Given user $i$ selects a charger of type $k$ at lot $j$, the first term defines (i) driving cost $v_{oj}$ to parking lot $j$ as well as $\mu_{j\delta}$ from charging spot at lot $j$ to destination $\delta$, where coefficient $\theta$ converts travel times to monetary values and (ii) waiting cost for chargers to become available, where $\theta'$ converts the expected average waiting time to monetary value. The second term indicates charging expense obtained by price $p^t_{jk}$ for the duration $n_{ijk}$ of parking, where $\alpha$ is a positive coefficient that weighs the impact of the charging expense in total costs imposed to EV users. 
Finally, the last term defines a penalty for staying longer than charging demand for the subsequent travel (i.e., overcharging penalty), where 
$\alpha'$ represents the overcharging penalty factor, and $\psi_i$ is the parking duration of user $i \in \mathcal{D}^t$.
Note that if user complies with the maximum allowed charging duration $n_{ijk}$ (i.e., $\psi_i = n_{ijk}$) there will be no penalty. We observe an additional cost if $\psi_i > n_{ijk}$ when user $i$ stays longer than their charging demand.
%
Constraints \eqref{eqref:cons1Main} ensure that each EV user can only choose one charger at $t$. Constraints \eqref{eqref:cons2Main} state that the charging opportunity is only available at parking lots selected by users. Constraints \eqref{eqref:cons3Main} update the SOC, i.e., $b_i^t$, of user $i \in \mathcal{D}^t$ at time $t$, where $\pi$ represents the charging rate in $Kw$ per time period.
Constraints \eqref{eqref:cons4Main} ensure that the SOC of user $i$ exceeds a certain threshold $Q_{i}$ before leaving the charging facility, given the subsequent travel plan of user $i$. And, constraints \eqref{eqref:cons5Main} ensure the capacity of facility with charger types $k$ at lot $j$ is not violated.

\begin{table}[H] 
 	\caption{Sets, decision/state variables, and parameters.}
 	\begin{center}
\vspace{-5mm} 			\scriptsize
 	\begin{tabular}{@{}ll}
 	\hline\hline
 	Sets\\
 			\hline
 			$\Gamma$
 			& Set of all time steps ${0, 1, \dots, T-1}$\\
 			$J$ & Set of parking spaces in a neighborhood \\
 			$K$ & Set of charger types \\
 			$\mathcal{O}$ & Set of EV travel origins\\
 			$\Delta$ & Set of EV travel destinations\\
			
 			Decision variables\\
 			\hline
 			$y_{ijk}$ & Binary; 1 if user $i$ uses type $k$ charger at lot $j$ \\ 
 			$n_{ijk}$ & Charging duration of user $i$ at $j$ with type $k$ charger\\
 			$a^t$ & Actions at each time $t$\\
			
 			States \& $t$-variants\\
 			\hline
 			$\widehat{\mathcal{J}}^t_{jk}$ & \# spots with type $k$ charger at $j$ first available at $t$\\
 			$\widehat{\mathcal{J}}^t$ & Spatial distribution of all newly realized spots at $t$\\
 			$\mathcal{J}^t_{jk}$ & \# spots with type $k$ charger at $j$ already available at $t$\\
 			$\mathcal{J}^t$ & Total \# charging spots already available at $t$\\
 			$\mathcal{J}^{t^+}$ & Total \# available charging spots at time $t$\\
 			$\mathcal{P}^t$ & Set of charging prices at time $t$\\
 			$p_{jk}^{t}$ & Charging price of a spot at $j$ with type $k$ charger at $t$\\
 			$\widehat{\mathcal{D}}^t$ & Set of users that first arrive at parking lots in time $t$\\
 			$\mathcal{D}^t$ & Set of existing users at $t$ before new EV arrivals\\
 			$\mathcal{D}^{t^+}$ & Set of available EV users at $t$, including new users\\
 			$\widehat{\mathcal{B}}^{t}$ &  Set of SOCs of EV users that first arrive at time $t$\\
 			$\mathcal{B}^{t}$ &  Set of SOCs of available EV users at time $t$\\
 			$b_{i}^{t}$ & SOC of EV user $i$ at time $t$\\
 			$\sigma_{jk}^{t}$ & Occupancy of type $k$ charger located in lot $j$ at $t$\\
			
 			$\Psi^{t}$ &  Set of parking durations of EV users at time $t$\\
 			$\widehat{\mathcal{L}}^t_{jk}$ & Expected avg waiting at $j$ with type $k$ charger at $t$\\
 			$\widehat{\mathcal{L}}^{t}$ &  Set of newly perceived waiting times at time $t$\\
 			$\mathcal{W}^t$ & Stochastic process with realization $\mathcal{W}^t(\omega)=\omega^t$\\ 
 			$\mathcal{S}^t$ & State of the system at each time $t$\\
 			$\Tilde{V}^{t}$ & Approximated value function\\
 			$\widetilde{\mathcal{S}}^{t,\,t'}$ & System state at time $t$ and $t'$ in look-ahead model\\
 			$\widetilde{\mathcal{A}}^{t,\,t'}$ & Set of actions at time $t$ and $t'$ in look-ahead model\\
 			$\widetilde{\Omega}^{t,\,t'}$ & Set of outcomes at time $t$ and $t'$ in look-ahead\\
 			$\mathcal{N}(\widetilde{S}^{t,\,t'})$ & \# visiting states $\widetilde{S}^{t,\,t'}$ in the tree search\\
 			${\mathcal{N}(\widetilde{S}^{t,\,t'},\,\widetilde{a}^{t,\,t'})}$ &
 			\# times a decision $\widetilde{a}^{t,\,t'}$ occurs when state is $\widetilde{S}^{t,\,t'}$\\
 			$\mathcal{C}_{b^{t,t'}_i}$ & Quadratic cone of $b^{t,t'}_i$\\
 			Parameters\\
 			\hline
 			$T$ & \# discrete time periods in the planning horizon\\
 			$v_{oj}$ & Driving cost from origin $o$ to charging spot at lot $j$\\
 			$\mu_{j\delta}$ & Driving cost from a spot at lot $j$ to destination $\delta$\\
 			$e_{jk}$ & Avg time of finding an available spot of type $k$ at $j$\\
 			$\beta_{jk}$ & Reaction to shared occupancy info; 1 for aware users\\ 
 			$c_{jk}$ & Capacity of charger type $k \in K$ at lot $j \in J$\\
 			$\theta$ & Coefficient to convert travel time to monetary value\\
 			$\theta'$ & Coefficient to convert waiting time to monetary value\\
 			$\alpha$ & Coefficient to weigh the impact of charging expense\\
 			$\alpha'$ & Overcharging penalty factor\\
 			$\pi$ & Charging rate in $Kw$ per time period\\
 			$M$ & A sufficiently large number\\
 			$Q_i$ & SOC threshold of user $i$ before leaving the facility\\
 			$\rho_z$ & Positive increasing parameter at iteration $z$\\
 		    $u_{jk}^{i,z}$ & Lagrangian multiplier of user $i$ for type $k$ charger at\\
 		    & lot $j$ at iteration $z$\\
 		    $N$ & \# iterations $n'$ of the MCTS algorithm\\
 			$H$ & Limited time horizon; threshold for tree expansion\\
 			$t'$ & Iteration in look-ahead; $t'=t,\,\dots,\,t+H-1$\\
 			$t''$ & Max allowed SOC update; $t'' \in [t', \mbox{max}\,\, n_{ijk})$\\
 			$\kappa$ & Max number of actions in the set of possible actions\\
 			$\tau$ & Iteration counter; update SOC before session ends\\
 			$\xi$ & \# iterations in the SH algorithm\\
 			$\gamma$ & Iteration counter in the SH algorithm\\
 			$\widetilde{\omega}^\gamma$ & Sample charging rate at iteration $\gamma$\\
 			$\iota$ & Balancing par. of exploration	\& exploitation in UCT\\
 			$\nu^{(1)}$, $\boldsymbol{\nu^{(2)}}$, & Lagrangian multipliers of constraints \eqref{eqref:cons1Main}-\eqref{eqref:cons5Main}\\
 			$\boldsymbol{\nu^{(3)}}$,$\boldsymbol{\nu^{(4)}}$, $\boldsymbol{\eta}$ &  \\
 			\hline
 		\end{tabular}
 		\label{table:definitions}
 	\end{center}
 	\end{table}

\vspace{-1mm}
\section{Solution Technique}\label{sec:Algorithm}
This section presents a hybrid solution technique that incorporates a (i) generalized Nash equilibrium, (ii) Monte Carlo tree search, and (iii) shooting heuristic into a stochastic look-ahead technique to find the optimal charging duration and spot assignment for each EV user over the planning horizon, given their subsequent travel plans and reactions toward charging prices and waiting time to get served. 
GNE distributes the proposed optimization problem \eqref{eqref:ObjFunMain}-\eqref{eqref:cons5Main} to EV user-level models given the strategy set of other EV users as exogenous information. Therefore, each EV user solves a lower dimension optimization problem more efficiently. Then, a MCTS algorithm with an embedded tree policy is applied to expand the tree search solely on promising branches. Moreover, a shooting heuristic is implemented to estimate the value of recently added tree nodes and determine the optimal charging schedules. 

\vspace{-5mm}
\subsection{Generalized Nash Equilibrium}\label{subsec:GNE}
\noindent The proposed problem in \eqref{eqref:ObjFunMain}-\eqref{eqref:cons5Main} suffers from a huge state space $S^t$, particularly when the number of EV users searching for available charging spots increases. 
%
%
Each EV user seeks to minimize its costs in finding an available charging spot, while the charger availability continuously changes due to the actions (a.k.a. strategies) of other users in selecting charging location and duration. The interaction among EV users (i.e., system players) represents a non-cooperative game-theoretic model in which players aim to maximize their own benefits individually without setting any agreement on optimal solutions.
A Nash equilibrium is defined as a condition if no player can maximize their reward by unilaterally changing their strategy. 
According to the literature, this problem can be formulated as a finite-dimensional variational inequality (VI) \citep{facchinei2007finite,harker1990finite}.
In particular, a GNE defines the proposed non-cooperative game as quasi-variational inequalities (QVIs) to address state-dependent strategy sets of other players \citep{facchinei2007generalized}.
The GNE distributes the proposed problem into univariate EV user level optimization models. Each user solves a convex optimization given the strategy set of other users as exogenous information. Thus, the problem can be solved very efficiently.

The strategy set of system players is subject to dynamic changes; e.g., un-occupied charging locations vary over time.
We apply a sequential penalty VI approach for QVIs, proposed by \cite{pang2005quasi}, to address the complexities involved in the dynamic strategy sets. We relax the general constraints \eqref{eqref:cons5Main} that include the decision variables of all users and incorporate them as a penalty term into the objective function \eqref{eqref:ObjFunMain} and then, solve a sequence of penalized VIs, as follows.
We let $\rho_z$ be a positive increasing parameter that satisfies $\rho_z < \rho_{z+1}$ and $u_{jk}^{i,z}$ as a weight vector for relaxed constraints $\boldsymbol{g}(\boldsymbol{y})$ for each user $i$ at lot $j$ with type $k$ charger, where $z$ is the iteration number of GNE. Then, problem \eqref{eqref:ObjFunMain}-\eqref{eqref:cons5Main} can be reformulated for each user $i$, given fixed $(\boldsymbol{y},\boldsymbol{n})$ for other users, as
%
%
	\begin{align}
		&\displaystyle{\operatornamewithlimits{\mbox{min}}_{\boldsymbol{y},\boldsymbol{n}}} \,\,\varphi^{t}_{i}(\boldsymbol{y},\boldsymbol{n})=\nonumber\\
        &\sum_{k \in K} \sum_{j \in J} \big( \theta (v_{oj} + \mu_{j\delta}) + \theta' \widehat{\mathcal{L}}^t_{jk} \big)y_{ijk} + \alpha\, p_{jk}^{t}\, n_{ijk} \nonumber\\
        &\epc\epc + \alpha'(\psi_{i} - n_{ijk}) \nonumber\\
        &\epc\epc + \rho_z^{-1} u_{jk}^{i,z}\, \exp({\rho_z g_{jk}^{i}(y_{-i,jk}, y_{i,jk})}), \label{eqref:ObjFun2}\\[2pt]
		& \mbox{subject to}\,\,\, \eqref{eqref:cons1Main}-\eqref{eqref:cons4Main}. \nonumber
	\end{align}
%
%
The objective function $\varphi^t_i$ aims to minimize the total costs (i.e., driving time to a charging spot, waiting time to get an available charger, and charging expense) of each user at each $t$, given an initial state $S^0$; i.e., $\operatornamewithlimits{\mbox{minimize}}_{\boldsymbol{y},\boldsymbol{n}} \varphi^{0}_{i}(\boldsymbol{y},\boldsymbol{n}) + \mathbb{E} \{\sum_{t \in \Gamma \setminus \{0\}} \operatornamewithlimits{\mbox{minimize}}_{\boldsymbol{y},\boldsymbol{n}}\,\, \varphi^{t}_{i}(\boldsymbol{y},\boldsymbol{n})  \}$.
%
%

Now, we propose a consensus-based coordination scheme to push the user-level solutions toward system-level optimality; i.e., to find near-optimal solutions. To this end, a penalty term $\boldsymbol{u} + \boldsymbol{\rho}\,\boldsymbol{g}(\boldsymbol{y})$ with associated Lagrangian multipliers $(\boldsymbol{\rho},\,\boldsymbol{u})$ is defined that represents the violation of constraints \eqref{eqref:cons5Main} based on other EV users' selected strategies.
The proposed approach updates the strategy set of users and estimates the Lagrangian multiplier $\boldsymbol{u}$ as follows. 
%
%
\begin{align}
    u_{jk}^{i,z+1} \equiv \text{max}\Bigl(0,\, u_{jk}^{i,z} + \rho_{z}\,g_{jk}^{i}(y_{-i,jk})\Bigr).\label{eqref:updateU}
\end{align}
EV users form a consensus on the location and duration of charging attempts by exchanging information on available spots over time.

\vspace{-4mm}
\subsection{Monte Carlo Tree Search}\label{subsec:MCTS}
\noindent This section applies a DP technique to obtain the minimum total cost for each user in an equilibrium condition. While the GNE procedure described in Section \ref{subsec:GNE} distributes the problem into user-level programs, the newly constructed problems defined for each user still experience a huge state space due to unknown uncertainties (i.e., stochastic charging rate) that occur over the planning horizon. The main source of complexity arises in the value function estimation, where we shall estimate the value of being in a particular state. Literature shows various techniques to tackle the intractability; for instance, \cite{godfrey2001adaptive,mirheli2018development,mirheli2020utilization,SnowPlow-2016}
exclude a state variable from the set of state variables and approximate the value function with the remaining ones. There are also studies that consider a limited number of time periods to simplify the look-ahead model in the dynamic program \cite[e.g.,][]{mirheli2019consensus,al2016information}. 
In this paper, the set of current waiting time $\mathcal{L}^t$ in the state space is obtained by equation \eqref{eq:waiting} using the occupancy $\sigma^t_{jk}$ at each parking lot $j$ for available users $i\in \mathcal{D}^t$. 
Then, the approximated value function $\Tilde{V}^{t+1}(\mathcal{B}^{t+1},\mathcal{J}^{t+1})$ only captures the impacts of available charging spots and SOC of each EV. The decision at each time $t$ is made as
%
\begin{align}
    &\mathcal{A^*}^t(S^t) = \displaystyle{\operatornamewithlimits{\mbox{argmin}}_{\boldsymbol{y},\boldsymbol{n} \in \mathcal{A}^t(S^t)}} (\varphi^{t}_{i}(\mathcal{B}^t,\mathcal{J}^t,\boldsymbol{y},\boldsymbol{n}) +\nonumber\\
    & \epc\epc\epc\epc\epc \Tilde{V}^{t+1}(\mathcal{B}^{t+1},\mathcal{J}^{t+1})) \label{eqref:appBellman}\\
    & \mbox{subject to}\,\,\, \eqref{eqref:cons1Main}-\eqref{eqref:cons4Main}.\nonumber
\end{align}

For notation simplicity, we let $a^t$ represent all actions (i.e., selected charging spots, selected charger types, and the length of each charging session) at each time $t\in\Gamma$. 
%
%
To ensure that equation \eqref{eqref:appBellman} is not unbounded, Theorem \ref{theorem1} denotes that the value function $\Tilde{V}^{t}_{a}(S^{t}_{a})$ is confined by a lower bound (i.e., the objective value of each user will not converge to $- \infty$), and actions $(\boldsymbol{y},\boldsymbol{n})$ are attainable at each time $t$. 

\begin{theorem}\label{theorem1}
Given the initial state $\mathbf{S}^0$, 
a lower bound for value function $\Tilde{V}^{t}_{a}(S^{t}_{a})$ will be defined as
    $
    \mathbb{E}[\operatornamewithlimits{min}_{\boldsymbol{a}} \varphi^{t}_{i}(\mathcal{B}^t,\mathcal{J}^t,\boldsymbol{a})],\, \forall t \in \Gamma$. 
\vspace{-5pt}
\end{theorem}
\begin{proof}\nobreak\ignorespaces
See Appendix. 
\end{proof}

\noindent We now define a tree policy to evaluate the available actions and approximate the value function at each time $t$. Afterward, a shooting heuristic is defined to estimate the value of the recently added node in the tree search. These steps fulfill the MCTS procedure, as follows.

\subsubsection{Tree Policy}\label{subsubsec:treePolicy}

\noindent 
The huge state space in the proposed problem, caused by the combination of feasible actions and exogenous information at each time $t$, imposes an exponential growth rate to the tree that makes the computation of value function at each branch interactable. Therefore, we apply an MCTS algorithm \citep{al2016information} with a look-ahead policy to (i) efficiently estimate the value function and (ii) effectively capture the uncertainties.

The MCTS algorithm generally consists of four steps \citep{browne2012survey,munos2014bandits}: selection, expansion, simulation, and back-propagation. In this study, we first assign each user $i$ to every parking lot $j$ with an available charging spot to generate the first level of the tree. We solve \eqref{eqref:cons1Main}-\eqref{eqref:cons4Main}, and \eqref{eqref:appBellman} in the selection step to (i) select the best charging allocation and duration actions from the pool of possible actions and (ii) estimate their value functions over iterations until we reach an expandable state. In each leaf node, if a user is still not assigned to any spot, all parking lots that are not fully occupied will be added to the action set, where each is represented by a newly added branch in the tree. We assume that EV users do not switch between parking lots after the assignment, where one state is added to the tree with the same charging spot of the same type. Then, in the simulation step, the value of the added state is calculated using SH to observe the impact of uncertain charging rates on SOC through subsequent time periods (see Section \ref{subsubsec:SH}). Finally, the back-propagation step updates the value functions of predecessor states based on the estimated value of recently added states. 


We acquire the exogenous information on unexpected average waiting time at the beginning of each time period $t\in \Gamma$, where the stochastic information is added to the tree to capture the uncertainties. In the look-ahead model, all variables are indexed with $t,\,t'$ to identify the time iteration $t$ in the main model and $t'=t,\,\dots,\,t+H-1$ in the look-ahead model, where $H$ represents a limited time horizon as a threshold for tree expansion (i.e., inner tree iterations). 
We apply a two-stage look-ahead model to (i) update the estimation of value function $\widetilde V^{t,t'}(\widetilde{S}^{t,t'})$ with the selected actions and (ii) compute the post-decision $\widetilde V^{t,t'}_a(\widetilde{S}^{t,t'}_a)$ value functions that include the effects of adding exogenous information, where $\widetilde{S}^t=\{\mathcal{B}^t,\,\mathcal{J}^t,\,\mathcal{P}^t\}$.
%
Additionally, at state $\widetilde{S}^{t,t'}$, we let $\widetilde{\mathcal{A}}^{t,\,t'}(\widetilde{S}^{t,t'})$ denote the set of decisions, where $\widetilde{\mathcal{A}}^{t,\,t'}_e(\widetilde{S}^{t,t'})$ defines the set of decisions explored in the tree at time $t'$ and its complement set $\widetilde{\mathcal{A}}^{t,\,t'}_u(\widetilde{S}^{t,t'})$ represents the unexplored decisions. 
%
Once the action space (i.e., charging spot, type, duration) is set for time $t'$, a sample of possible outcomes $\widetilde{\omega} \in \widetilde{\Omega}^{t,\,t'+1}(\widetilde{S}^{t,t'}_a)$ will be generated and fed into \eqref{eq:waiting} to compute the expected average waiting time of available users in each lot $j$, which is unknown prior to time $t$. For the possible outcomes, we let $\widetilde{\Omega}^{t,\,t'+1}(\widetilde{S}^{t,t'}_a)$ represent all possible random events that can take place at time $t'+1$, where $\widetilde{\Omega}^{t,\,t'+1}_e(\widetilde{S}^{t,t'}_a)$ and $\widetilde{\Omega}^{t,\,t'+1}_u(\widetilde{S}^{t,t'}_a)$ denote the explored and unexplored possible outcomes, respectively.

The proposed stochastic MCTS framework has the computational budget of $N$ iterations (see more details in \citep{mirheli2020utilization}). At each time $t$, the current state $S^t$ is captured to generate a state $\widetilde{S}^{t,\,t'}$ as a root node of the tree, generate the MCTS algorithm, build a look-ahead model to estimate the value functions $\widetilde{V}^{t+1}$, and return the vector of near-optimal actions $\mathcal{A}^{*^t}$ at time period $t$.
%
%
In the selection step, there is a trade-off between exploiting the high-reward states and exploring the states frequently ignored during the search, until we reach a threshold of sufficient possible actions $\kappa$. Therefore, in the selection step, we follow upper confidence-bounding (UCT) for trees, defined in \cite{browne2012survey}, as follows
%
\begin{align}
   &\widetilde{a}^{*^{t,\,t'}}=\nonumber\\
   &\operatornamewithlimits{argmax}_{\widetilde{a}^{t,t'} \in \widetilde{\mathcal{A}}^{t,t'}_e (\widetilde{S}^{t,t'})}\Bigl(-(\widetilde{\varphi}^t(\widetilde{S}^{t,t'},\widetilde{a}^{t,t'})\,+\widetilde{V}^{t,t'}_a(\widetilde{S}^{t,t'}_a))+\nonumber\\
   &\iota\, \sqrt{\frac{\ln \mathcal{N}(\widetilde{S}^{t,t'})}{\mathcal{N}(\widetilde{S}^{t,t'},\,\widetilde{a}^{t,t'})}}\Bigr), \label{UCT}
\end{align}
%
%
where $\iota$ is an adjustable parameter to balance exploration and exploitation, $\mathcal{N}(\widetilde{S}^{t,\,t'})$ represents the number of visiting states $\widetilde{S}^{t,\,t'}$, and ${\mathcal{N}(\widetilde{S}^{t,\,t'},\,\widetilde{a}^{t,\,t'})}$ identifies the number of times a decision $\widetilde{a}^{t,\,t'}$ is taken from state $\widetilde{S}^{t,\,t'}$ during the tree search process. Once decisions $\widetilde{a}^{t,\,t'}$ are made, the state of the system will be updated, i.e., $\widetilde{S}^{t,\,t'}_a$, where we add a sample realization of exogenous information to reach the next pre-decision state $\widetilde{S}^{t,\,t'}=S^{\mathcal{T},a}(\widetilde{S}^{t,\,t'}_a,\,\widetilde{\mathcal{W}}^{t,\,t'+1})$.  
Here, $S^{\mathcal{T},a}$ represents the transition function between the evolution of each two consecutive state variables. 
In the simulation step, we develop an SH-based approach to obtain an initial estimation for the newly added tree nodes. In this step, we first generate a sample path $\widetilde{\omega} \in \widetilde{\Omega}^{t,\,t'}(\widetilde{S}^{t,\,t'})$ to determine the level of information provided for the users at time $t$. More details follow.
\subsubsection{Shooting Heuristic}\label{subsubsec:SH}
\noindent This section describes an SH to estimate the value of the recently added node in the tree search. After adding the node by allocating the charging facility or identifying the charging duration at each time period $t$, we generate a sample path of charging rates for each EV that will be followed in the subsequent time periods. Accordingly, the SOC of each EV is updated by the average value of sample charging rates. An infeasible solution is reported when the EV cannot receive the required charge by the time of leaving the facility (i.e., if $\psi_i < n_{ijk}$ is observed). We define $\xi$ as the number of iterations in SH. Figure \ref{fig:SH} depicts all attempts generated by sample charging rates over $\gamma = 1,\dots,\xi$ iterations.
\vspace{-10pt}
\begin{figure}[H]
	\begin{center}
	\includegraphics[height=1.8 in]{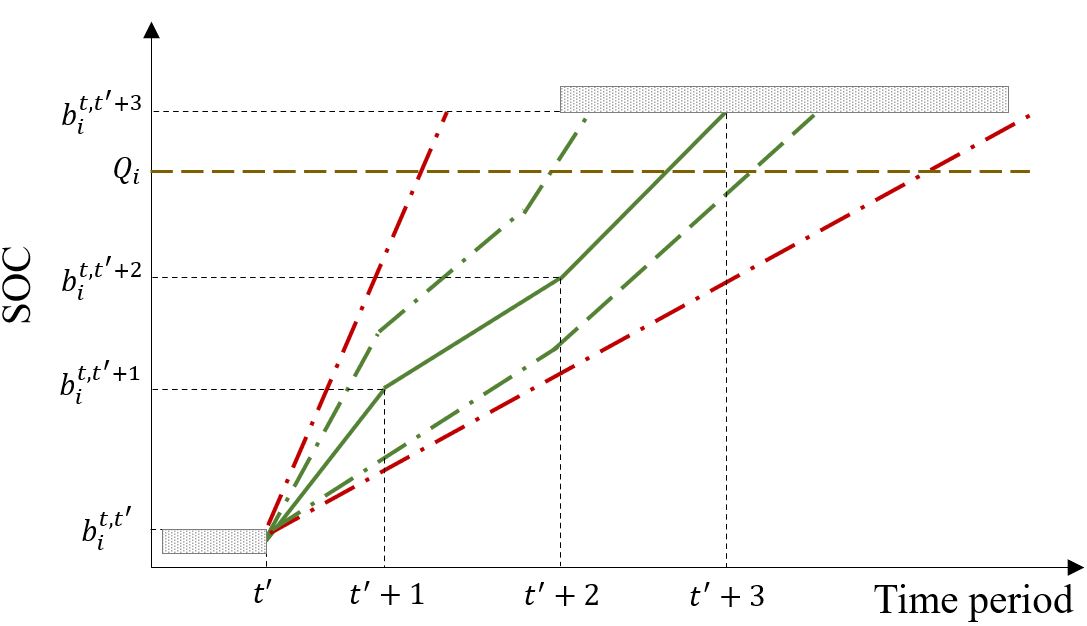}
	\vspace{-20pt}
		\caption{Shooting heuristics with stochastic charging rates.}
		\label{fig:SH}
	\end{center}
\end{figure}
%
Given $b^{t,t'}_i$ as the SOC of user $i \in \mathcal{D}^t$ at time $t$ and look-ahead time $t'$, the sample path $\widetilde{\omega}^\gamma$ denotes the charge amount offered by the charging facility for $(t, t')$ at iteration $\gamma$. The SOC of each EV (i.e., $b^{t,t'}_i$) is updated based on $\widetilde{\omega}^\gamma$ values until we reach the parking duration $\psi_i$ of user $i\in \mathcal{D}^t$ (i.e., users $i$ leaves the charging spot after $\psi_i$ duration). The proposed heuristic is described in Algorithm \ref{SH} as follows that is embedded in the MCTS framework. 
\begin{algorithm}[H]
 	\caption{The SH procedure.}\label{SH}
 	\small

 	\begin{algorithmic}[1]
 		\Statex \textbf{procedure} {SH(${b}^{t,t'}_i$)}{}
 		\Statex Collect $b^{t,t'}_i, n_{ijk}, Q_i $
 		\Statex Set $\xi$, $b^{t,t'}_{i,\gamma} \gets b^{t,t'}_i$, and $\tau = 1$
 		\Statex \textbf{while} $\gamma<\xi$
 		\Statex Generate $\widetilde{\omega}^\gamma$ for the required session length $n_{ijk}$
 		\Statex \hspace{3mm} \textbf{for} $\tau \le n_{ijk}$
 		\Statex \hspace{10mm} Update $b^{t,t'+\tau}_{i,\gamma} \gets b^{t,t'+\tau}_{i,\gamma} + \pi(k+1) y_{ijk}$
 		\Statex \hspace{3mm} \textbf{end} 
 		\Statex \hspace{3mm} \textbf{if} $b^{t,t'+n}_{i,\gamma}<Q_i$
 		\Statex \hspace{10mm} \textbf{remove} $b^{t,t'+n}_{i,\gamma}$ from the generated SOCs
 		\Statex \hspace{3mm} \textbf{end}
 		\Statex \textbf{end}
 		\Statex Average $b^{t,t'+\tau}_{i} \gets b^{t,t'+\tau}_{i,\gamma}$ for all $\tau$
 	\end{algorithmic}
\end{algorithm}

Figure \ref{fig:flowchart} shows the general framework for GNE with an embedded consensus-based coordination scheme and incorporated MCTS-SH to solve the optimization problem \eqref{eqref:cons1Main}-\eqref{eqref:cons4Main}, and \eqref{eqref:appBellman} under the charging rate uncertainties. 

\vspace{-10pt}
\begin{figure}[H]
	\begin{center}
	\includegraphics[height=3.1 in]{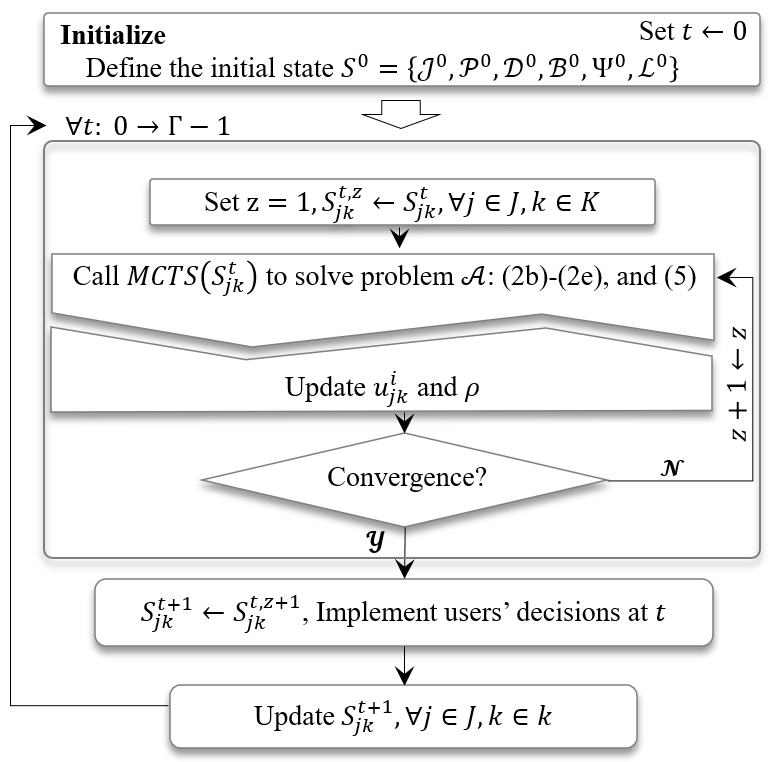}
		\vspace{-20pt}
		\caption{GNE-MCTS general framework.
		}
	\label{fig:flowchart}
	\end{center}
\end{figure}

Proposition \ref{prop1} denotes that SH develops a feasible set of SOC levels for an arbitrary path of charging rates, given feasible values of $b^{t,t'}_i$. Definition \ref{def1} supports Proposition \ref{prop1} by introducing the quadratic cone of SOC values.

\begin{definition}\label{def1}
 The quadratic cone of $b^{t,t'}_i$ defines the set of sample SOCs based on variant charging rates following the initial SOC as 
 $\mathcal{C}_{b^{t,t'}_i} = \{b^{t,t''}_i|  b^{t,t''-1}_i + \pi y_{ij,0}\, \le\, b^{t,t''}_i \le\, b^{t,t''-1}_i + 2 \pi y_{ij,1}\}$,
where $t'' \in [t', \mbox{max}\,\, n_{ijk})$ is generated from $\widetilde{\omega}^\gamma$.
\end{definition}

\begin{proposition}\label{prop1}
Given a feasible starting SOC value of $b^{t,t'}_i$, the quadratic cone $\mathcal{C}_{b^{t,t'}_i}$ is not empty if and only if $\psi_{i}^{-1}(Q_{i}-b^{t,t'}_i) \in [\pi, 2\pi]$.
\end{proposition}
\begin{proof}
See Appendix. 
\end{proof}

\begin{lemma}
 Given any $\delta' \ge 0$ and a feasible starting SOC value of $b^{t,t'}_i$, if quadratic cone $\mathcal{C}_{b^{t,t'+\delta '}_i}$ is not empty, then $\mathcal{C}_{b^{t,t'}_i}$ is not empty.
\end{lemma}

The distribution of a central optimization problem into EV user level optimizations may introduce infeasible solutions due to the relaxation of the user connection constraints. 
Given charging location and duration decision pair $(\boldsymbol{y}^z,\boldsymbol{n}^z)$ is bounded, Proposition \ref{prop2} shows that the solution obtained by GNE with an embedded consensus-based coordination scheme \cite[e.g.,][]{mirheli2019consensus,niroumand2020joint,9564622,9294741} converges to the solution of \eqref{eqref:ObjFunMain}-\eqref{eqref:cons5Main} for sufficiently large $z$. 


\begin{proposition}\label{prop2}
Given a non-empty $\mathcal{C}_{b^{t,t'}_i}$, suppose the relaxed constraints $\boldsymbol{g}(\boldsymbol{y})$ is continuously differentiable and convex for each charging location and duration decision pair $(\boldsymbol{y},\boldsymbol{n})$. Let $y_{ijk}^{\infty}$ be the convergence subsequent at iteration $z$. Then, $y_{ijk}^{\infty}$ is the solution to GNE with an embedded consensus-based coordination scheme for the sequence of $\rho_z$ and $u_{jk}^{i,z}$, if the following hold:
\begin{subequations}
\begin{align}
    &\sum_{k \in K} \sum_{j \in J} \left\{\eta_{jk}^{i} \nabla g_{jk}^{i}(y_{ijk}^{\infty},y_{ijk}^{\infty}) - M \nu_{jk}^{(2)} + \nu_{jk}^{(3)} \pi (k+1) \right\} \nonumber\\
    &\epc\epc\epc\epc\epc 
     + \nu^{(1)}|J||K| = 0, \label{eqref:prop1_1}\\
    &\sum_{k \in K} \sum_{j \in J} \nu_{jk}^{(2)}  - \sum_{j \in J} \nu_{j}^{(4)} \pi |K| = 0, \label{eqref:prop1_2}\\ 
    & \eta_{jk}^{i} \,g_{jk}^{i}(y_{ijk}^{\infty}) = 0,\, \forall j \in J,\, k \in K, \label{eqref:prop1_3}\\
    & \nu^{(1)} (\sum_{k \in K} \sum_{j \in J} y_{ijk} - 1) = 0, \label{eqref:prop1_4}\\
    & \nu_{jk}^{(2)} (n_{ijk} - M\, y_{ijk} - 1) = 0,\, \forall j \in J,\, k \in K, \label{eqref:prop1_5}\\
    & \nu_{jk}^{(2)}(b_{i}^{t}+\pi(k+1) y_{ijk} - b_{i}^{t+1}) = 0, \forall j \in J, k \in K, \label{eqref:prop1_6}\\
    & \nu_{jk}^{(3)} (Q_{i} - b_{i}^{t} - \sum_{k \in K} \pi (k+1) n_{ijk}) = 0,\,\forall j \in J, \label{eqref:prop1_7}\\
    & \eta_{jk}^{i} \ge 0,\, \forall j \in J,\, k \in K, \label{eqref:prop1_8}\\
    & \nu^{(1)} \ge 0, \label{eqref:prop1_9}\\
    & \nu_{jk}^{(2)},\,\nu_{jk}^{(3)},\, \forall j \in J,\, k \in K, \label{eqref:prop1_10}\\
    & \nu_{j}^{(4)} \ge 0,\, \forall j \in J, \label{eqref:prop1_11} 
\end{align}
\end{subequations}
where $\nu^{(1)}$, $\boldsymbol{\nu^{(2)}}$, $\boldsymbol{\nu^{(3)}}$, $\boldsymbol{\nu^{(4)}}$, and $\boldsymbol{\eta}$ denote the Lagrangian multipliers of constraints \eqref{eqref:cons1Main}-\eqref{eqref:cons5Main}. 
\end{proposition}

\begin{proof}
See Appendix. 
\end{proof}

\section{NUMERICAL EXPERIMENTS}\label{sec:numericalresults}
\noindent The solution technique proposed in Section \ref{sec:Algorithm} is applied to a hypothetical and a real-world case study to assess the computational performance and solution quality. The methodology is coded in Java and run on a desktop computer with octa-core 3.1 GHz CPU and 64 GB of memory. A Poisson distribution is used to generate the initial charging demand pattern (i.e., users $i\in \mathcal{D}^t$) for five different time-of-days in a business day, i.e., early AM, AM peak, mid-day, PM peak, and evening. Additionally, we assume that the length of charging follows an exponential distribution \cite{alizadeh2013scalable}. It is also assumed that the charging rate at each facility follows a normal distribution with the mean of nominal rate $\boldsymbol{\lambda}$ (i.e., 6.2 $Kw$ for slow chargers and 12.5 $Kw$ for fast chargers per time period) and the standard deviation of 10$\%$. The value of $\pi$ is set to the charging rate of slow chargers, i.e., 6.2 $Kw$. We call the CPLEX library in JAVA to solve the optimization model in \eqref{eqref:cons1Main}-\eqref{eqref:cons4Main}, and \eqref{eqref:appBellman} at each iteration.

\vspace{-3mm}
\subsection{Hypothetical Dataset}\label{subsec:dataset}
To proposed model in \eqref{eqref:ObjFunMain}-\eqref{eqref:cons5Main} and hybrid solution framework is applied to a hypothetical network shown in Figure \ref{fig:hypoNetwork}. The network dataset includes 18 nodes and 58 links, where charging facilities are deployed on nodes 6,7,15, and 16 with 5 chargers at each node. EV travel origin is assumed to be from node 1, while nodes 8 and 11 are the destination nodes. 
We have assumed a planning horizon from 8 AM to 5:30 PM with 30 $min$ time periods. The average vehicle arrivals over different time-of-days are respectively assumed to be 20, 25, 20, 15, and 30 for early AM, AM peak, mid-day, PM peak, and evening for a medium demand level. 
The low and high demand levels are assumed to be half and twice the medium demand, respectively.
%
It is assumed that EV users start their travels from origin $o \in \mathcal{O}$, stop in parking lot $j \in J$, 
and leave chargers (i) after reaching to a sufficient SOC for their subsequent trips or (ii) when their maximum allowed charge duration is reached.
We assume parameter $\alpha$ in the objective function to be 1.
Additionally, the weight $\alpha'$ of charging expense for staying longer at charging facilities is set to 0.1.
\vspace{-7pt}
	\begin{figure}[H]
		\begin{center}
	\includegraphics[height=1.7 in]{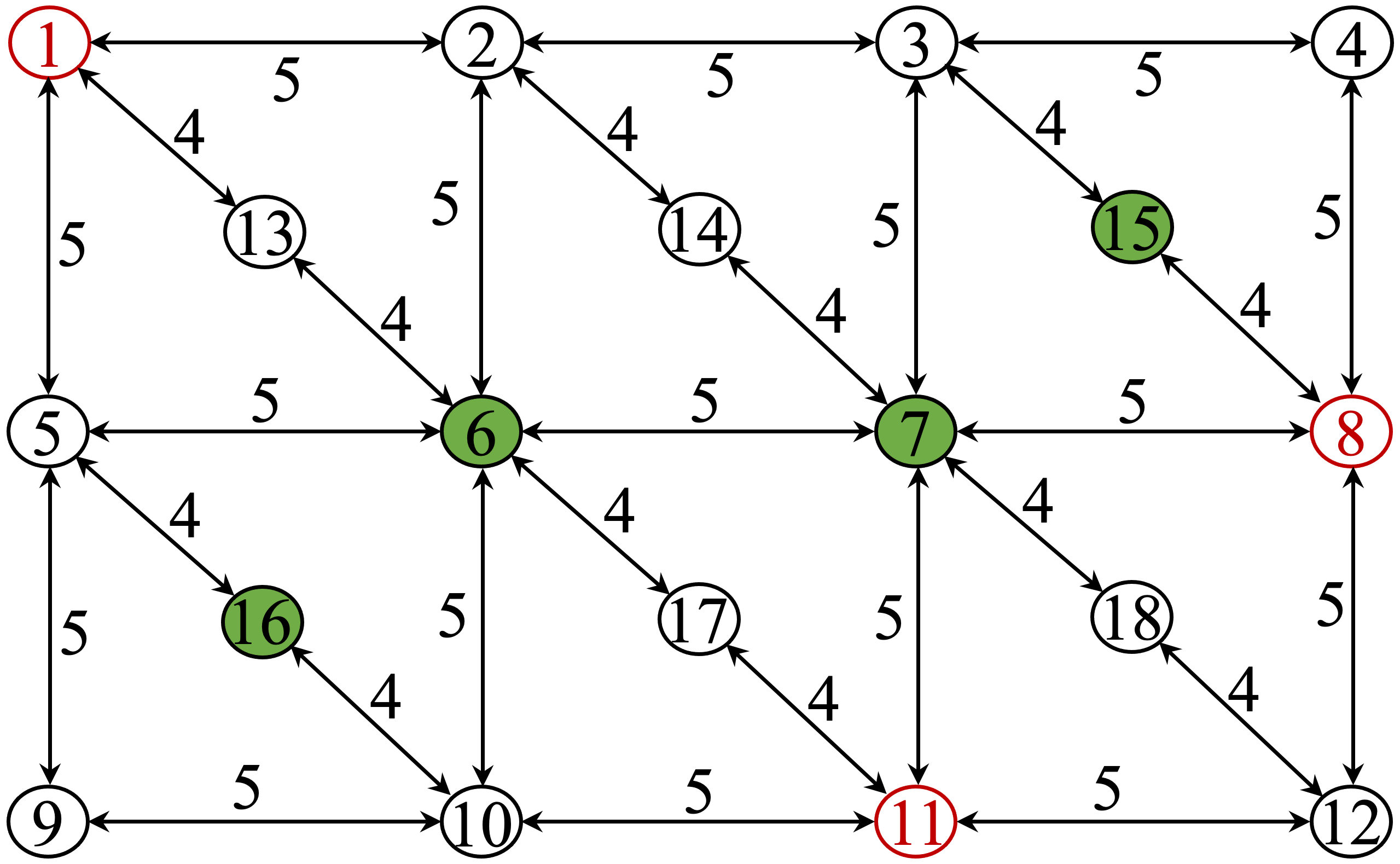}
	\vspace{-10pt}
	\caption{Hypothetical network.}
	\label{fig:hypoNetwork}
		\end{center}
	\end{figure}
\vspace{-10pt}
%
%
%

Figure \ref{fig:hypo_obj} presents the sum of objective values \eqref{eqref:ObjFunMain} of all EVs choosing to charge over all time periods. We can observe a decreasing trend in the total objective value with respect to the GNE iterations that indicates EV users keep forming consensus on optimal solutions. 
We observe considerable reductions in the objective value when the number of iterations increases from 1 to 4. However, the solution does not improve significantly after iteration 4. As computational time has a direct relationship with the number of iterations to reach consensus, we use the iteration 4 results as a base for the remainder of the algorithm,
since no significant improvement (i.e., less than 0.5\%) is observed in the objective value afterward. 
\vspace{-10pt}
	\begin{figure}[H]
		\begin{center}
	\includegraphics[height=1.9 in]{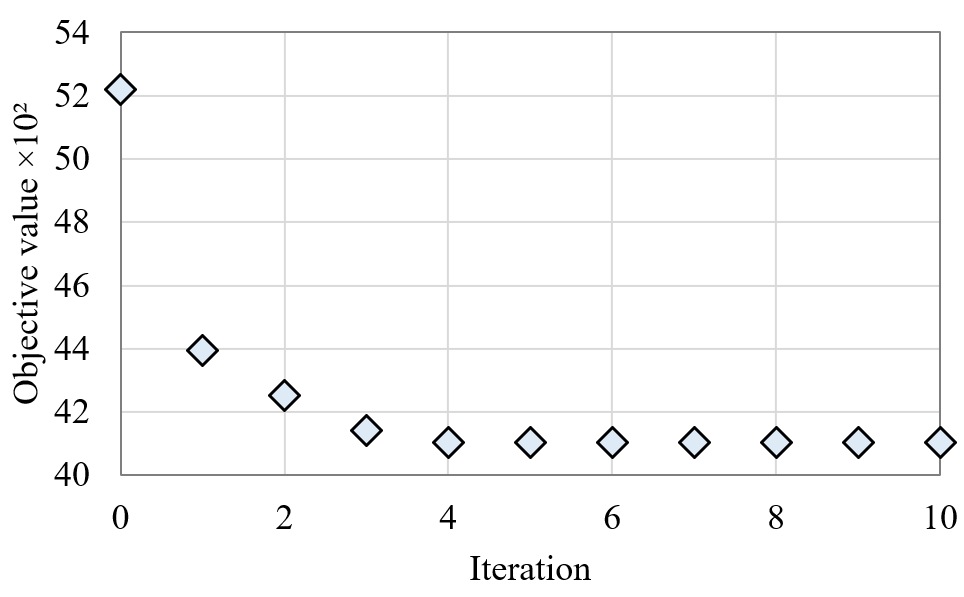}
	\vspace{-10pt}
			\caption{The change in the objective value (\$).}
	\label{fig:hypo_obj}
		\end{center}
	\end{figure}
\vspace{-10pt}
We study the impact of the information exchange among EV users toward reaching consensus in the proposed scheduling technique. To this end, we assess the charging facility occupancy over iterations of the algorithm. Figure \ref{fig:DiffOccu} presents the marginal occupancy of a facility located in the parking lot at node 7 in consecutive iterations. We can observe that the changes in occupancy are significant at the beginning of the algorithm, while the marginal occupancy reaches a steady state with zero changes toward the end due to the users' agreement on the optimal scheduling actions. 
\vspace{-10pt}
 \begin{figure}[H]
 	\begin{center}
 	\includegraphics[height=1.42 in]{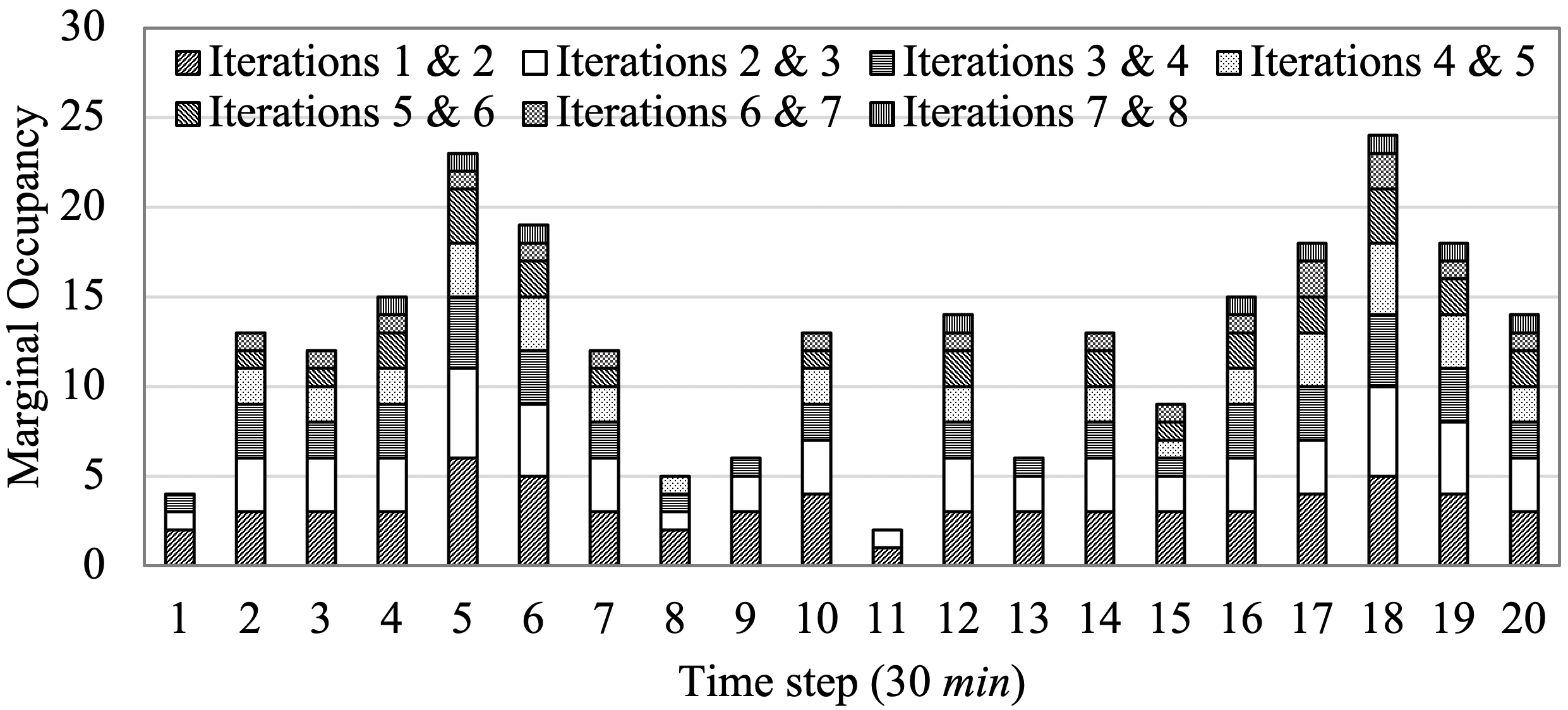}
 	\vspace{-20pt}
 		\caption{Marginal occupancy of consecutive iterations in the hypothetical dataset.}
 	\label{fig:DiffOccu}
 	\end{center}
 \end{figure}
\vspace{-10pt}
	
Figure \ref{fig:hypo_MCTS} indicates the frequency of selecting charging methods considering various average charging rates (i.e., $k=1,2$) for 100 iterations of MCTS over the network. EV users tend to choose (i) low-occupancy charging facilities with more available service time periods left for charging per spot due to the stochasticity involved in charging rates at facilities or (ii) switch to faster charger types to secure enough SOC before their subsequent trips.  
\vspace{-10pt}
\begin{figure}[H]
	\begin{center}
	\includegraphics[height=1.7 in]{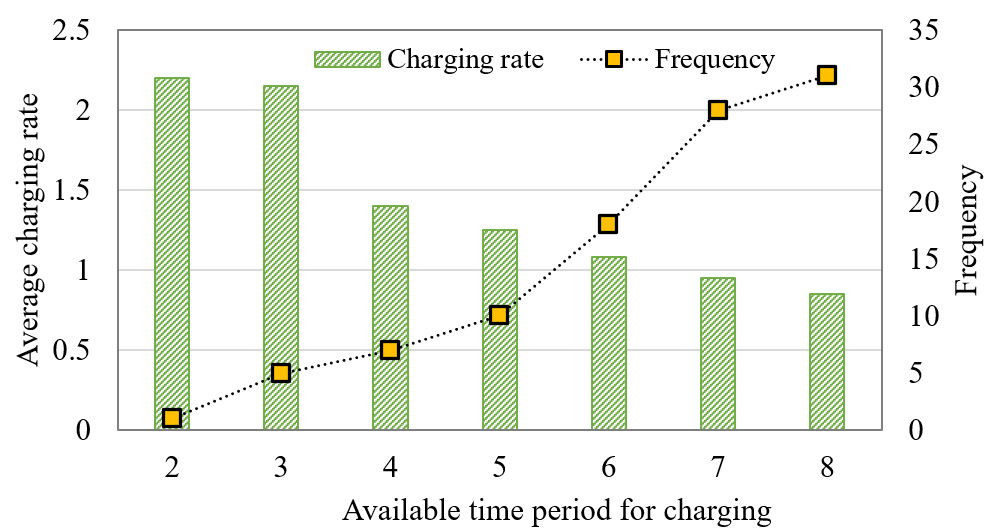}
	\vspace{-20pt}
		\caption{Average charging rate ($Kw$) and frequency for various available charging time periods in MCTS.}
	\label{fig:hypo_MCTS}
	\end{center}
\end{figure}

\vspace{-8mm}	
\subsection{Real-world dataset}\label{Real dataset}
\noindent The proposed methodology is applied to a real-world case study in North Carolina. 
The network includes 42 nodes, 451 links, and 13 parking lots with charging facilities in North Carolina State University campus, as shown in Figure \ref{fig:network_real}. 
The figure indicates the facility locations that EV travelers (i.e., faculty, staff, students, and visitors) tend to charge on campus. 
The origins and destinations are located in Raleigh, Durham, and Chapel Hill. 
The values for time-to-monetary value coefficient $\theta$, charging cost coefficient $\alpha$, and overcharging penalty factor $\alpha'$ are set to 0.1, 10, and 0.1, respectively.
We assume that the number of chargers at each charging facility is 10. 
The demand in this dataset is distributed over time as 140, 210, 170, 120, and 230 for early AM, AM peak, mid-day, PM peak, and evening time-of-days. Similar to the hypothetical dataset, the low and high demand levels are assumed to be half and twice of the medium demand level, respectively. With a 20-time period dynamic scheduling scheme, the real-world dataset includes 1,120,560 decision variables, given an average of 174 users in each time period. 
To generate results for the real-world case study, we run the algorithm up to 10 iterations for GNE since the solution will not improve significantly afterwards (i.e., the change in the objective value sum of all users is below 0.5\%).

Figure \ref{fig:objValue_real} represents the value of travel cost and charging expense in the objective function over iterations. As illustrated, EV users choose the nearest parking lot to their destinations due to a lack of information from other users in the first iterations of the algorithm. However, the exchange of information over iterations improves the perception of EV users about the occupancy and demand of each charging facility for which they may experience slightly higher travel costs. On the other hand, we can observe a decreasing trend in the charging expenses due to the newly perceived penalty costs when the occupancy of all parking lots is updated. Thus, EV users tend to stay in a charging facility to secure just enough SOC for their next trips and avoid the overcharging penalty. 
%
\begin{figure}[H]
	\begin{center}		\includegraphics[height=1.8 in]{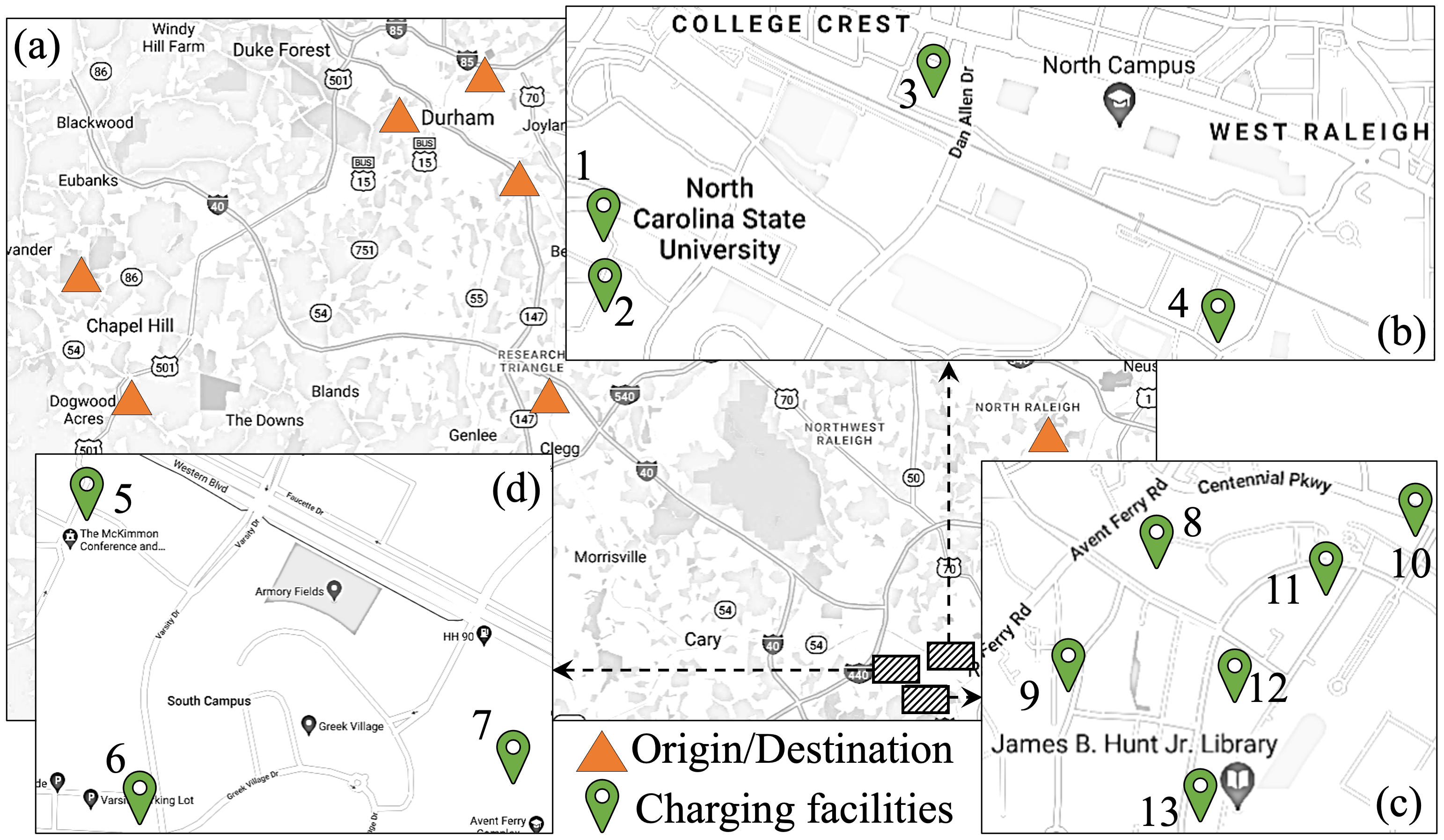}
	\vspace{-7pt}
		\caption{(a) Travel  origins/destinations, NC; (b),(c),(d) Charging facility locations on North Carolina State University campus. [Map source: Google, accessed November 25, 2020].}
	\label{fig:network_real}
	\end{center}
\end{figure}
\vspace{-20pt}

\begin{figure}[H]
	\begin{center}
	\includegraphics[height=1.55 in]{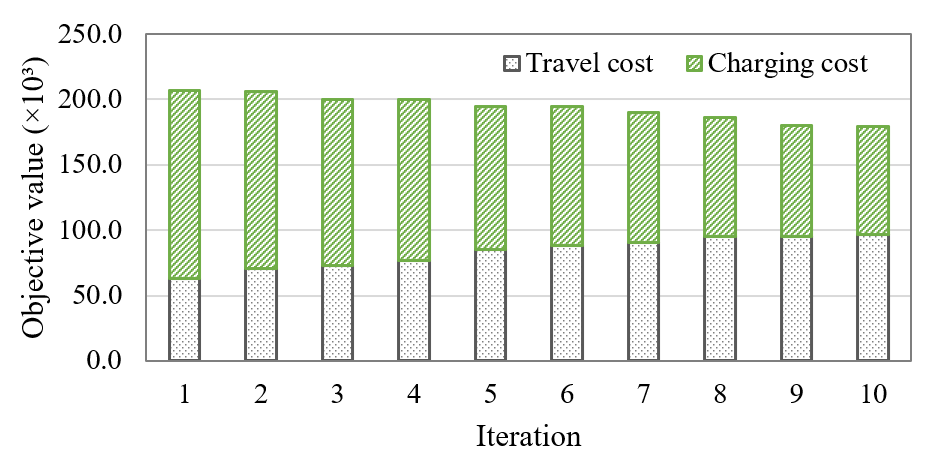}
	\vspace{-20pt}
	\caption{Objective value (\$) over 10 GNE iterations.}
	\label{fig:objValue_real}
	\end{center}
\end{figure}
\vspace{-8pt}
Figure \ref{fig:STDcharging} presents the average charging rate as well as the standard deviation for EV users who are realized at time period 18 (i.e., 4:30-5:00 PM) in the medium demand case. 
The standard deviation captures the reaction of users in choosing the charger types. EV users begin to select slow charging spots as they are less expensive. However, overcharging attempts lead to extra waiting costs for incoming users. Therefore, they will choose fast chargers to avoid the overcharging penalty.
Based on the exchanged information over 10 iterations of GNE, EV users will charge at facilities with higher charging rates as more appealing facilities (e.g., with lower rates and consequently lower costs) continue to get occupied by other users. Although the average charging rate approximately follows an increasing trend, its standard deviation fluctuates until all users receive updated information on charging spot occupancy.

Figure \ref{fig:SOC} illustrates the average SOC of EV users who select charging at facility 4 in medium demand case. As observed, users (with sufficient SOC to reach facility 4 before their SOC falls below 1 $Kwh$) need to stay in the charging spots at node 4 to be able to fulfill their upcoming trips. Note that the maximum distance to final destinations is assumed to be 48 $miles$, which requires 16.7 $Kwh$ for an average EV (i.e., with 2.91 $mile/Kwh$) in time period 3 (9:00-9:30 AM) and time period 12 (1:30 - 2:00 PM).
\vspace{-7pt}
\begin{figure}[H]
	\begin{center}
	\includegraphics[height=1.85 in]{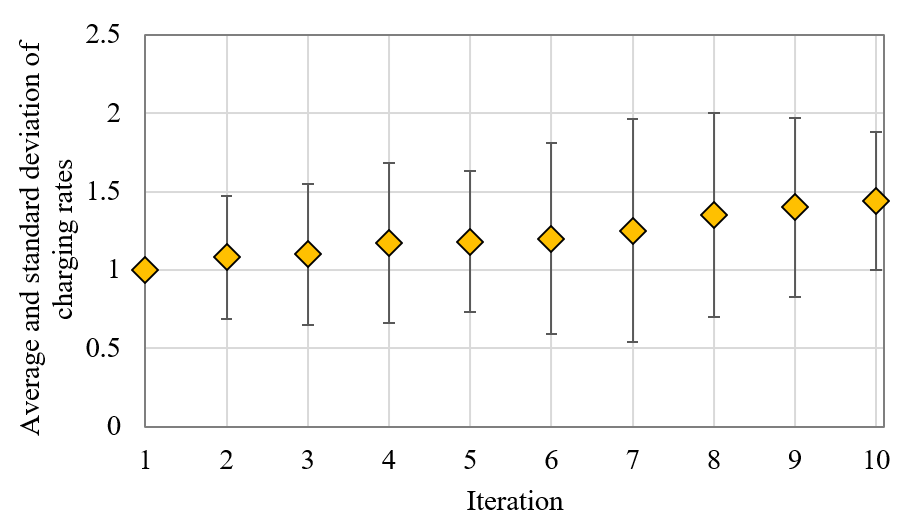}
		\vspace{-22pt}
		\caption{Average and standard deviation of charging rates ($Kw$) over 10 GNE iterations.}
	\label{fig:STDcharging}
	\end{center}
\end{figure}

%
\vspace{-15pt}
\begin{figure}[H]
	\begin{center}		\includegraphics[height=1.75 in]{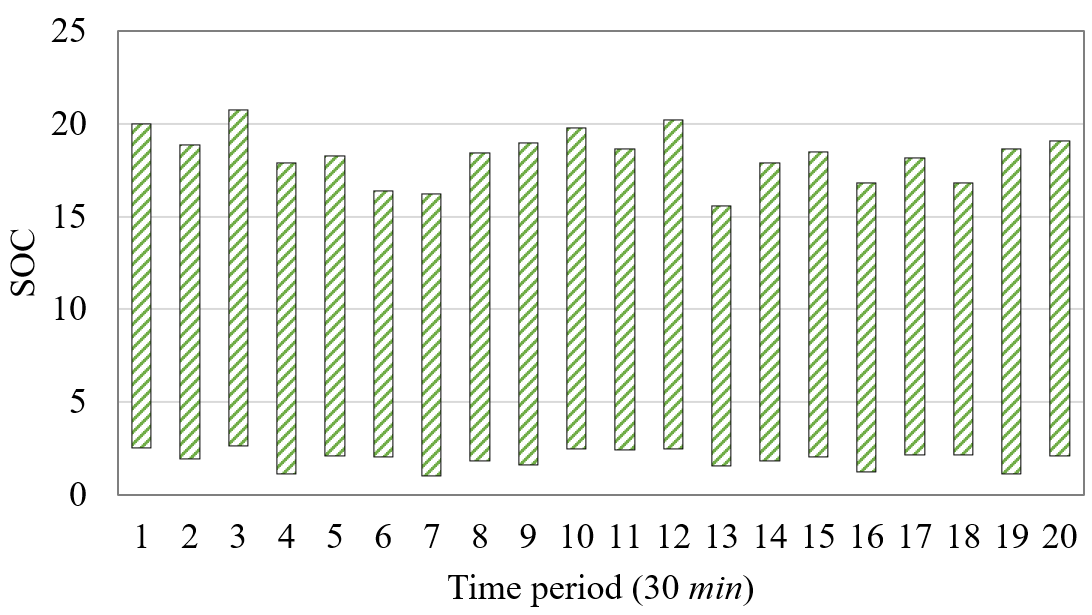}
		\vspace{-10pt}
		\caption{SOC changes ($Kwh$) of EVs during charging at facility 3.}
		\label{fig:SOC}
	\end{center}
\end{figure}

\vspace{-15pt}
Similar to the hypothetical case study, we analyze the occupancy over iterations. Figure \ref{fig:DiffOccu_real} presents changes in the occupancy of the charging facility located in the lot at node 4 for the medium demand scenario. We can observe significant changes in each two consecutive iterations over the first iterations of the algorithm. Similarly, the exchange of information among users leads to a steady state with zero changes toward the last iteration. The CPU time of this scenario is 3.1 $hr$. The results are obtained in real-time as the optimal solutions are found within the duration of each time period that is 30 $min$. 

We have also conducted a sensitivity analysis to evaluate the impact of the overcharging penalty factor $\alpha'$ on the objective value \eqref{eqref:ObjFunMain}, considering (1) total travel cost and (2) combined charging and penalty costs, imposed to each user $i$. As Figure \ref{fig:DiffOccu_real2} indicates, 
higher values of $\alpha'$ change the charging pattern: some EV users leave charging facilities with a sufficient charge to avoid the overcharging penalty. Therefore, we see a slight increase in the travel cost corresponding to movements from charging spots to regular parking spots. 
Moreover, we notice that penalty cost significantly decreases as users tend to avoid the increasing penalty. 
\vspace{-7pt}
\begin{figure}[H]
	\begin{center}
	\includegraphics[height=1.85 in]{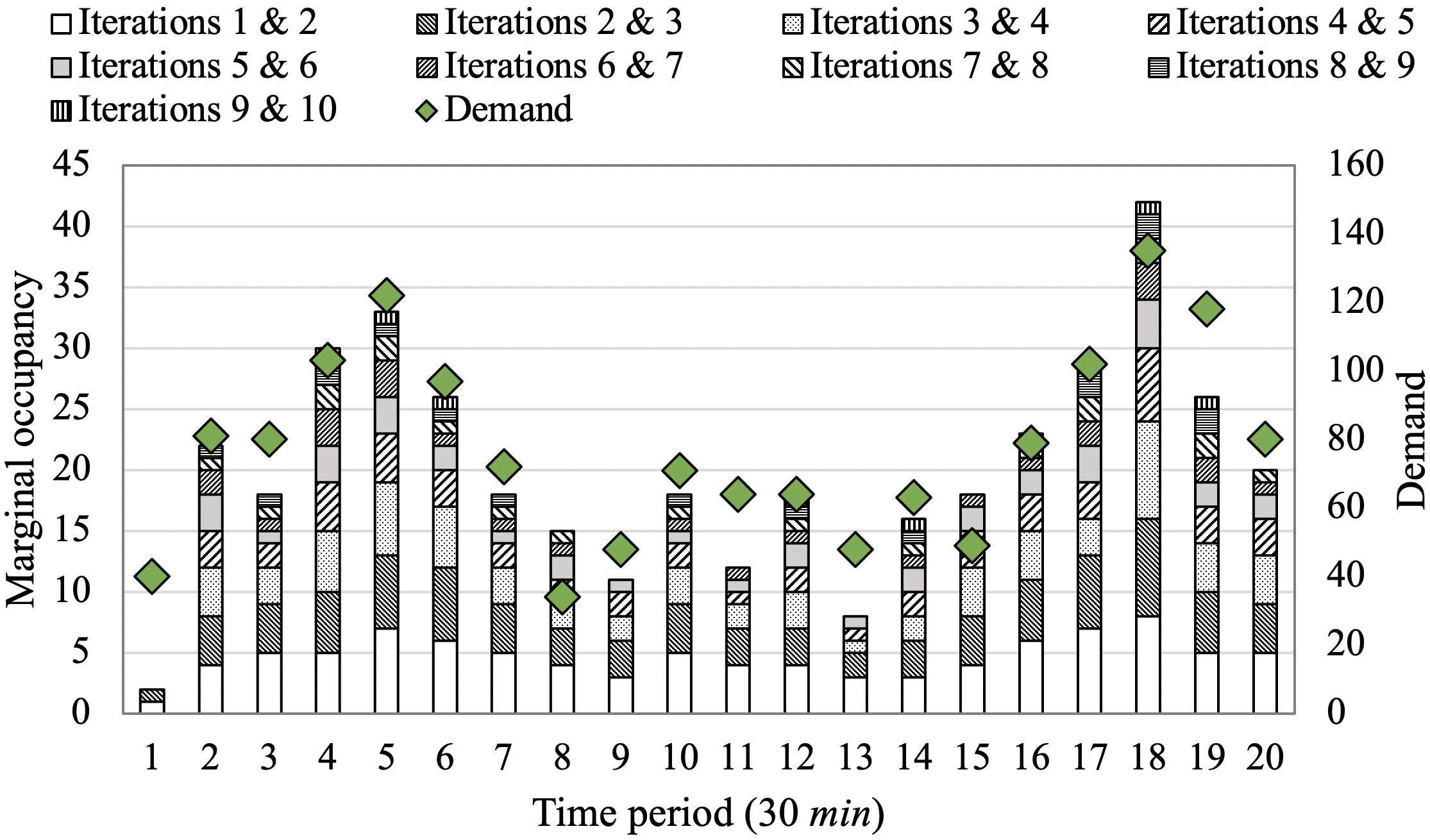}
		\vspace{-20pt}
		\caption{Marginal occupancy of consecutive iterations.}
	\label{fig:DiffOccu_real}
	\end{center}
\end{figure}
\vspace{-20pt}

\begin{figure}[H]
	\begin{center}
	\includegraphics[height=1.3 in]{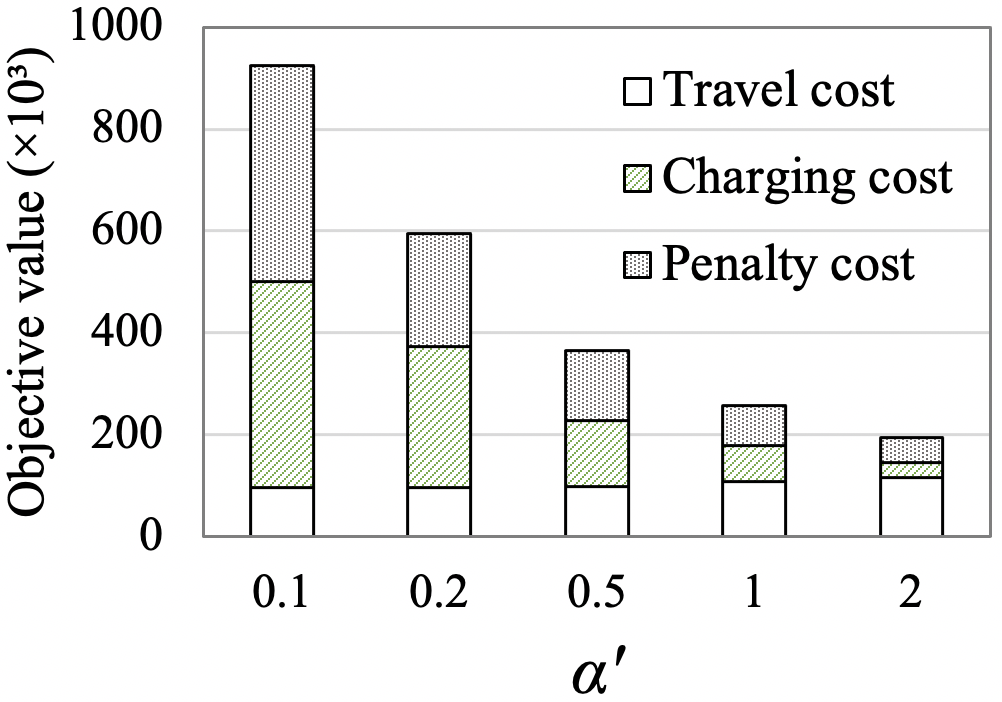}
	\vspace{-10pt}	
	\caption{Objective value ($\$$) versus coefficient of penalty term ($\alpha'$).}
	\label{fig:DiffOccu_real2}
	\end{center}
\end{figure}
\vspace{-13pt}
The solutions of the consensus-based approach are compared to those of a benchmark strategy that prioritizes EV users upon their arrivals. Let us describe how the benchmark works. At each time period, the users are assigned to available charging spots following a first-come-first-serve scheme given the updated system state (e.g., charger availability, charging demand). If the capacity of a charging facility is reached, EVs will form a queue to get served without a look-ahead consideration. The upcoming charging demand and charger availability are ignored in the user assignments to chargers at each time period. This implies that the queue length does not affect the user assignments. 
Figure \ref{fig:priority} and Table \ref{table:priority} present objective values obtained by our consensus-based (with look-ahead policy) versus priority-based scheduling after 10 iterations over the entire horizon. 
%
\vspace{-20pt}
\begin{figure}[H]
	\begin{center}
	\includegraphics[height=1.31 in]{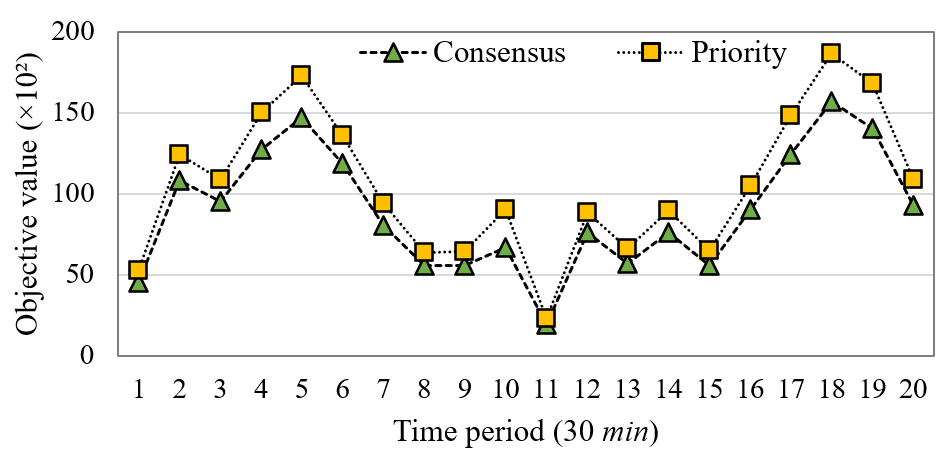}
	\vspace{-7pt}
		\caption{The comparison of objective value ($\$$) in consensus-based and prioritized scheduling schemes.}
	\label{fig:priority}
	\end{center}
\end{figure}
%
\begin{table}[H]
	\caption{Comparison with benchmark solutions.}
	\begin{center}
		\small
	\begin{tabular}{M{1.5cm} M{1.5cm} M{1.5cm} M{1.5cm}}
			\hline\hline
			 & Consensus-based & Priority-based & Difference $(\%)$\\
			\hline
			Objective value $(\$)$ & $1.79\times 10^5$ & $2.10\times 10^5$ & 17.58 \\
			CPU time $(hr)$ & 2.83 & 2.07 & -26.85 \\
			\hline
		\end{tabular}
	\end{center}
	\label{table:priority}
\end{table}

\vspace{-7mm}
\section{Conclusions} \label{sec:conclusion}
\noindent This study develops a charge scheduling scheme for EV users over time and space, given facilities with different charging types and limited capacities. EV users obtain information on the charger availability at each facility. Additionally, EVs' SOC, subsequent travel plans, and departure times are shared with the network operator upon their arrival at the facility. The objective is to determine the optimal charging schedule for each user that minimizes the (i) total travel costs and charging expenses and (ii) overcharging penalty. The problem is formulated as a DP model that captures the state of the system (i.e., EV arrivals at facilities, SOC, available chargers, pricing scheme, and waiting times) and takes actions using a stochastic look-ahead technique. The model is first distributed to user-level optimizations using a GNE approach that relaxes the coupling constraints connecting users. A consensus-based coordination scheme is incorporated into the GNE procedure to push the user-level solutions toward system-level optimality and find near-optimal solutions. Then, a Monte Carlo tree search algorithm with an embedded tree policy and an SH is proposed to efficiently capture the uncertainties and approximate the value function.
The tree policy evaluates the available actions and estimates the value functions over time, while the SH predicts the value of recently added nodes to the tree, under uncertain charging rates, to determine the near-optimal charging schedules. Numerical experiments on a hypothetical and a real-world dataset confirm the solution quality and efficiency of the proposed methodology. The results show that EV users tend to choose the charging facilities with more available service time periods left to reduce the uncertainty involved in the charging rate and guarantee enough SOC before leaving the charging facility. 
It will be very interesting to study the computational performance of the proposed hybrid solution technique based on a range of problem sizes considering different network instances and EV charging demand. Decomposition techniques, e.g., column generation based approaches \cite{Asya-Conf-2021,hajibabai2019patrol}, could be used to further improve the computational efficiency.
Another future research direction would be the inclusion of agency-level decisions (e.g., charging pricing decisions, infrastructure decisions) into the framework and solving the problem as a bi-level program \cite[e.g.][]{Hierarchical-2022,Pricing-Conf-2020,zhou2005generalized,nourinejad2016equilibrium,hajibabai2014joint}.
It will also be interesting to gather real-world EV charging demand data to analyze the charger utilization with various scheduling schemes given user preferences over time.
\renewcommand{\baselinestretch}{1.0}
 \vspace{-2mm}
\small
\bibliographystyle{IEEEtran}
\bibliography{main}

\vspace{-3mm}
\section*{Appendix} \label{proofs}

\noindent \textbf{Proof of Theorem 1:}
\begin{proof}\nobreak\ignorespaces
As denoted, $\Tilde{V}^{t}_{a}(S^{t}_{a}) = \mathbb{E}(\varphi^{t}_{i}(\mathcal{B}^t,\mathcal{J}^t,\boldsymbol{a}))$. Thus, based on the definition of minimization we have:
$\Tilde{V}^{t}_{a}(S^{t}_{a}) = \mathbb{E}( \varphi^{t}_{i}(\mathcal{B}^t,\mathcal{J}^t,\boldsymbol{a})) \ge \mathbb{E}(\operatornamewithlimits{min}_{\boldsymbol{a}} \varphi^{t}_{i}(\mathcal{B}^t,\mathcal{J}^t,\boldsymbol{a}))$.
\end{proof}

\noindent \textbf{Proof of Proposition 1:}
\begin{proof}\nobreak\ignorespaces
We first prove the necessity. If there exists a feasible SOC path (i.e., $\psi_i \ge n_{ijk}$), then $b^{t,t'}_i + \pi n_{ij,0} \le Q_{i} \le b^{t,t'}_i + 2\pi n_{ij,1}$. 
In other words, the charging threshold can be reached within the parking duration based on the feasible charging rates. 
To explore the sufficiency, given $\psi_{i}^{-1}(Q_{i}-b^{t,t'}_i) \in [\pi, 2\pi]$, we can obtain that $b^{t,t'-1}_i + \pi n_{ij,0} \le b^{t,t'}_i \le b^{t,t'-1}_i + 2\pi n_{ij,1}$. For $n_{ij,1} \ge 1$ and know that $2\pi y_{ij,1} \le 2\pi n_{ij,1}$, we can conclude that $\mathcal{C}_{b^{t,t'}_i}$ is not empty.
\end{proof}

\noindent \textbf{Proof of Proposition 2:}
\begin{proof}\nobreak\ignorespaces
For all sufficiently large $z$, we claim multipliers $\boldsymbol{\nu}^{(1)}_{z},\,\boldsymbol{\nu}^{(2)}_{z},\,\boldsymbol{\nu}^{(3)}_{z}$, and $\boldsymbol{\nu}^{(4)}_{z}$ exist such that:\\

\begin{subequations}
\begin{align}
    & \sum_{k \in K} \sum_{j \in J} \Bigl( \big( \theta (v_{oj} + \mu_{j\delta}) + \theta' \widehat{\mathcal{L}}^t_{jk} \big)y_{ijk}^{z} + \alpha p_{jk}^{t} n_{ijk}^{z} + \nonumber\\
    &\epc\epc \alpha'(\psi_{i} - n_{ijk}^{z}) 
    + \rho_z^{-1} u_{jk}^{i,z}\exp({\rho_z g_{jk}^{i}(y_{i,jk}^{z}, y_{i,jk}^{z})}) -\nonumber\\
    &\epc\epc M \nu_{jk}^{(2)} + \nu_{jk}^{(3)} \pi (k+1) \Bigr) +  \nu^{(1)}|J||K|  = 0,\label{eqref:proof1}\\
    & \sum_{k \in K} \sum_{j \in J} \Bigl( \big( \theta (v_{oj} + \mu_{j\delta}) + \theta' \widehat{\mathcal{L}}^t_{jk} \big)y_{ijk}^{z} + \alpha p_{jk}^{t} n_{ijk}^{z} +\nonumber\\
    &\epc\epc \alpha'(\psi_{i} - n_{ijk}^{z}) +\nu_{jk}^{(2)} \Bigr) - \sum_{j \in J} \nu_{j}^{(4)} \pi |K| = 0, \label{eqref:proof2}\\
    & \nu^{(1)} (\sum_{k \in K} \sum_{j \in J} y_{ijk}^{z} - 1) = 0, \label{eqref:proof3}\\
    & \nu_{jk}^{(2)} (n_{ijk} - M\, y_{ijk}^{z} - 1) = 0,\, \forall j \in J,\, k \in K, \label{eqref:proof4}\\
    & \nu_{jk}^{(2)} (b_{i}^{t} + \pi (k+1) y_{ijk}^{z} - b_{i}^{t+1}) = 0,\, \forall j \in J,\, k \in K, \label{eqref:proof5}\\
    & \nu_{jk}^{(3)} (Q_{i} - b_{i}^{t} - \sum_{k \in K} \pi (k+1) n_{ijk}^{z}) = 0,\,\forall j \in J. \label{eqref:proof6}
\end{align}
\end{subequations}
Constraints \eqref{eqref:proof1}-\eqref{eqref:proof6} hold given \eqref{eqref:prop1_1}, \eqref{eqref:prop1_9}-\eqref{eqref:prop1_11} and along with the following condition:
\begin{align}
&\sum_{k \in K} \sum_{j \in J} ( - M \nu_{jk}^{(2)} + \nu_{jk}^{(3)} \pi (k+1))+ \nu^{(1)}|J||K| = 0\nonumber.
\end{align}
The aforementioned conditions are equivalent to Mangasarian—Fromowitz constraint qualification \citep{outrata2013nonsmooth} at solution $(\boldsymbol{y}^z, \boldsymbol{n}^z)$ and thus, the claim holds.
\end{proof}
\end{document}